\documentclass[12pt]{amsart}
\usepackage{graphicx}
\usepackage{here}

\newtheorem{theorem}[subsection]{Theorem}
\newtheorem{lemma}[subsection]{Lemma}

\newtheorem{question}[subsection]{Question}
\newtheorem{remark}[subsection]{Remark}

\theoremstyle{definition}
\newtheorem{definition}[subsection]{Definition}

\title{Minimal coloring number for $\mathbb{Z}$-colorable links}

\author{Kazuhiro Ichihara}
\address{Department of Mathematics, 
College of Humanities and Sciences, Nihon University,
3-25-40 Sakurajosui, Setagaya-ku, Tokyo 156-8550, Japan}
\email{ichihara@math.chs.nihon-u.ac.jp}

\author{Eri Matsudo}
\address{Graduate School of Integrated Basic Sciences, Nihon University,
3-25-40 Sakurajosui, Setagaya-ku, Tokyo 156-8550, Japan}
\email{cher16001@g.nihon-u.ac.jp}

\date{\today}

\begin{document}

\maketitle

%%%%%%%%%%%%%%%%%%%%%%%%%%%%%%%%%%%%%%%%%%%%%%%%%%%%%%%%%%%%%%%%%%%%%%%%

\begin{abstract}
For a link with zero determinants, a $\mathbb{Z}$-coloring is defined as a generalization of Fox coloring. 
We call a link having a diagram which admits a non-trivial $\mathbb{Z}$-coloring a \textit{$\mathbb{Z}$-colorable link}. 
The \textit{minimal coloring number} of a $\mathbb{Z}$-colorable link is the minimal number of colors for non-trivial $\mathbb{Z}$-colorings on diagrams of the link.
We give sufficient conditions for non-splittable $\mathbb{Z}$-colorable links to have the least minimal coloring number.
\end{abstract}

\section{Introduction}

In \cite{Fox}, Fox introduced one of the most well-known invariants for knots and links, 
now called \textit{the Fox $n$-coloring}, or simply $n$-coloring for $n\ge 2$. 

For a link $L$, if the determinant of $L$ is 0, 
then it is shown that $L$ admits no non-trivial $n$-coloring for any $n\ge 2$. 
Please refer to \cite{IchiharaMatsudo} for the definition of the determinant of a link for example. 
In that case, $L$ admits a $\mathbb{Z}$-coloring defined as follows.

\begin{definition}\label{def1}
Let $L$ be a link and $D$ a regular diagram of $L$. 
We consider a map $\gamma:\{$arc of $D\}\rightarrow \mathbb{Z}$. 
If $\gamma$ satisfies the condition 
$2\gamma(a)= \gamma(b)+\gamma(c)$ 
at each crossing of $D$ 
with the over arc $a$ and the under arcs $b$ and $c$, 
then $\gamma$ is called 
a \textit{$\mathbb{Z}$-coloring} on $D$. 
A $\mathbb{Z}$-coloring which assigns the same color to all the arcs of the diagram 
is called the \textit{trivial $\mathbb{Z}$-coloring}. 
A link is called \textit{$\mathbb{Z}$-colorable} if it has a diagram admitting a non-trivial
$\mathbb{Z}$-coloring. 
\end{definition}

The links illustrated in Figure \ref{L8n6ex} and \ref{pretzelex} are examples those are $\mathbb{Z}$-colorable.
Throughout this paper, we adopt the names of links as those given in \cite{linkinfo}.

\begin{center}
\begin{figure}[H]
\includegraphics[width=8cm,clip,bb = 0 0 476 355]{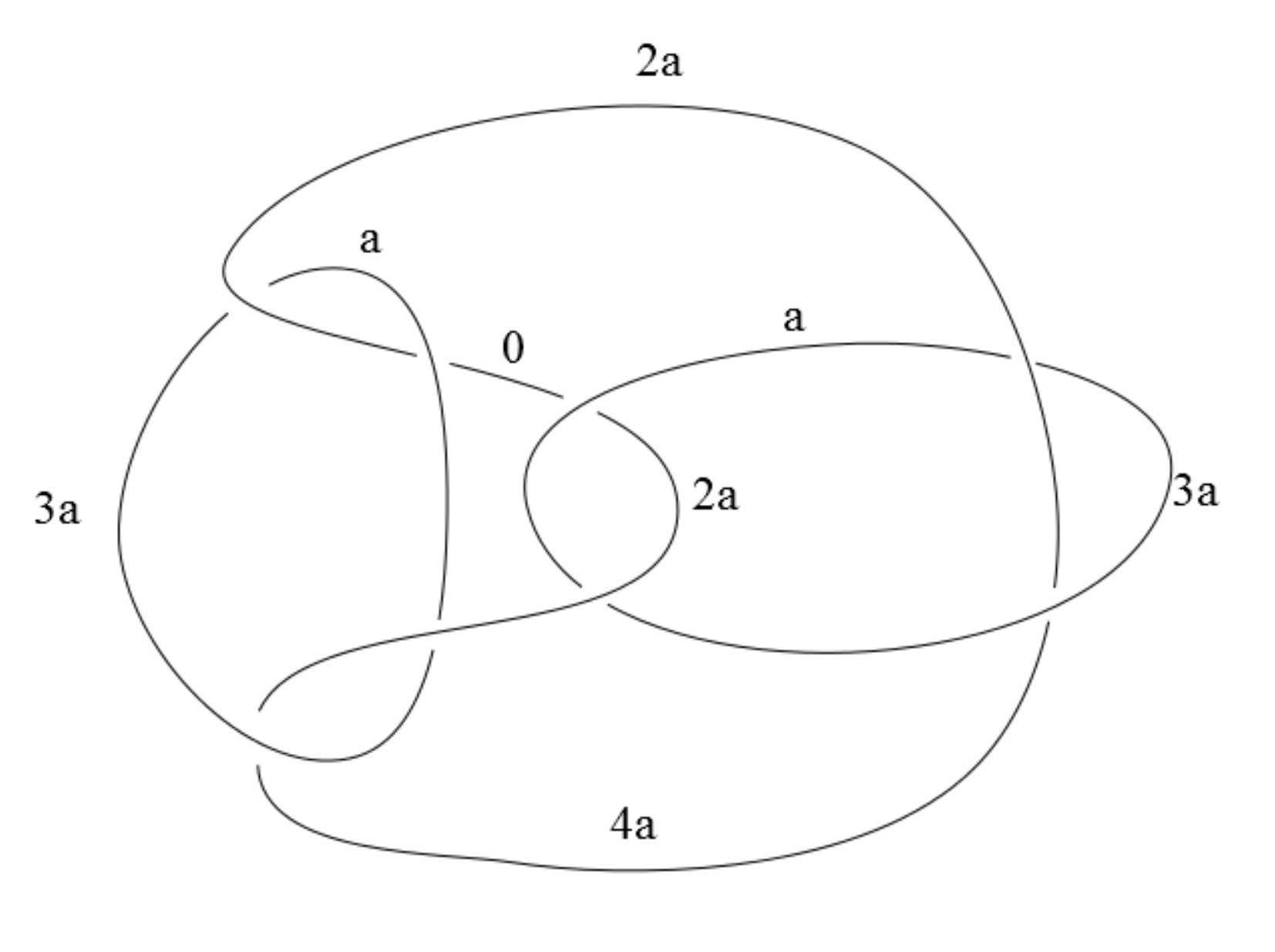}
\caption{$L8n6 \ (a\ge1)$}\label{L8n6ex}
%\end{figure}
%\begin{figure}[H]
\includegraphics[width=10cm,clip, bb= 0 0 600 450]{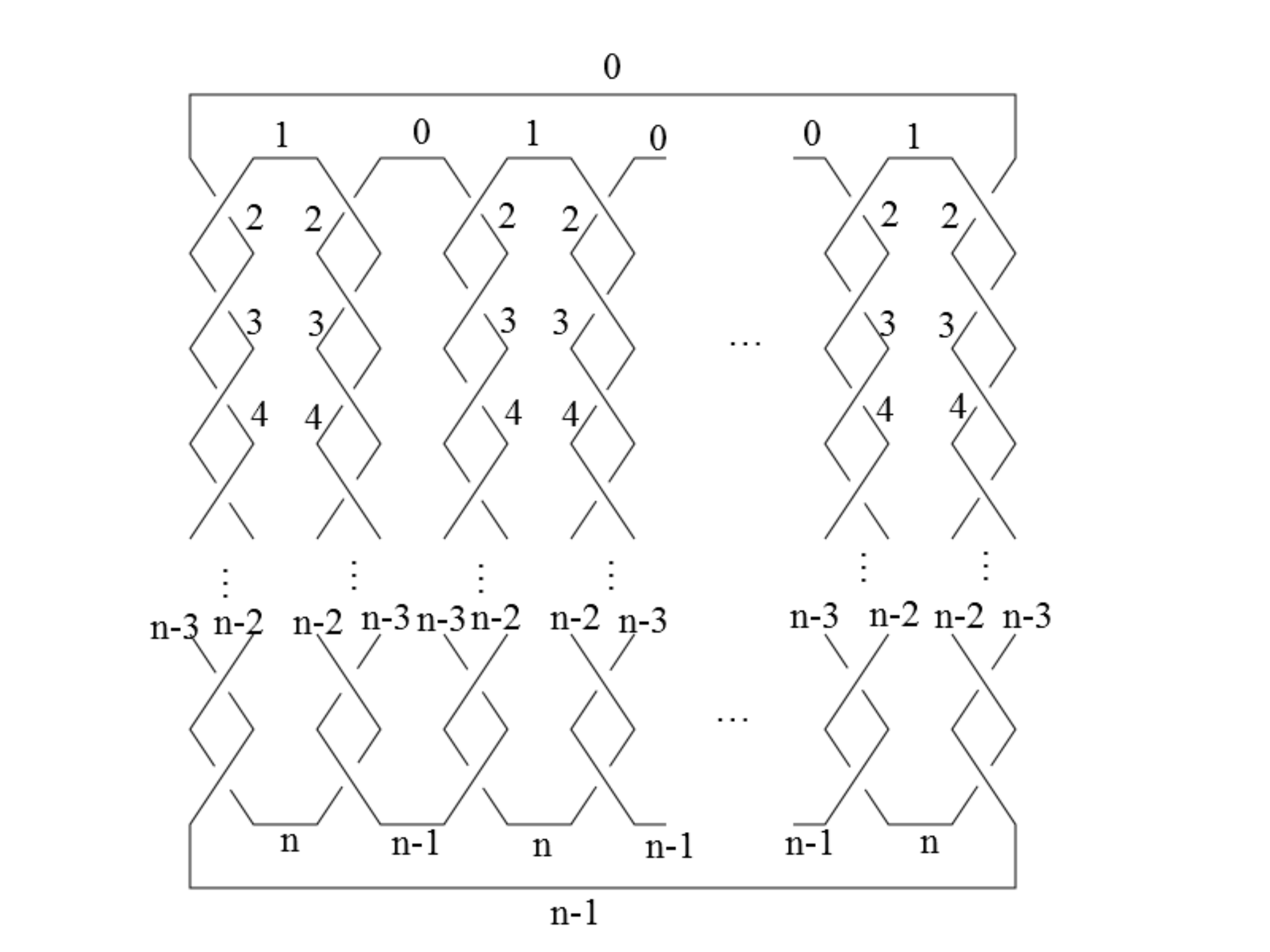}
\caption{Pretzel link $P(n,-n,\cdots,n,-n)$ for $n\ge1$}\label{pretzelex}
\end{figure}
\end{center}

In \cite{HararyKauffman}, 
Harary and Kauffman defined the minimal coloring number for links with Fox colorings. 
Here we define 
the minimal coloring number for $\mathbb{Z}$-colorable links as a generalization.

\begin{definition}\label{def2}
Let us consider the cardinality of the image of 
a non-trivial $\mathbb{Z}$-coloring on a diagram 
of a $\mathbb{Z}$-colorable link $L$. 
We call the minimum of such cardinalities among 
all non-trivial $\mathbb{Z}$-colorings on all diagrams of $L$ 
the \textit{minimal coloring number} of $L$, 
and denote it by $mincol_\mathbb{Z}(L)$. 
\end{definition}

In this paper, we give sufficient conditions for non-splittable $\mathbb{Z}$-colorable links to have the least minimal coloring number. 
First, after some preliminaries in Section \ref{sec:pre}, for a non-splittable $\mathbb{Z}$-colorable link $L$, 
we show that $mincol_\mathbb{Z}(L)\geq4$ in Section \ref{sec:4colors}. 
Next, we introduce a \textit{simple} $\mathbb{Z}$-coloring in Section \ref{sec:simple}, and show that if a link $L$ admits a simple $\mathbb{Z}$-coloring, then $mincol_\mathbb{Z}(L)=4$. 
Also, in the case that a link $L$ admits a $\mathbb{Z}$-coloring by five colors, then we have $mincol_\mathbb{Z}(L)=4$, which is shown in Section \ref{sec:five}.

%%%%%%%%%%%%%%%%%%%%%%%%%%%%%%%%%%%%%%%%%%%%%%%%%%%%%%%%%%%%%%%

\section{Preliminaries}\label{sec:pre}

Throughout this paper, 
for a $\mathbb{Z}$-coloring $\gamma$, 
we represent that \textit{the colors at a crossing are $\{a|b|c\}$} 
if the over arc is colored by $b$ and the under arcs are colored by $a$ 
and $c$ by $\gamma$ at the crossing. 

In this section, we prepare some basic 
properties of $\mathbb{Z}$-colorings.

\begin{lemma}\label{lem0}
For any $\mathbb{Z}$-colorable link, 
there exists a $\mathbb{Z}$-coloring $\gamma$ 
such that \textit{Im}$(\gamma)=\{0,a_1,a_2,\cdots,a_n\}$ 
with $a_i >0$ $(i=1,2,\cdots,n)$ 
for some positive integer $n$.
\end{lemma}

\begin{proof}
We assume that there exists 
a $\mathbb{Z}$-coloring $\gamma'$ such that \textit{Im}$(\gamma')=$
$\{\alpha,b_1,b_2,\cdots,b_n\}$ 
with $b_i >\alpha$ $(i=1,2,\cdots,n)$ for some positive integer $n$.  
We consider a map $\gamma:\{$arc of the diagram$\}\rightarrow \mathbb{Z}$ with \textit{Im}$(\gamma)=\{0,a_1,a_2,\cdots,a_n\}$ obtained by setting $a_i=b_i-\alpha$. 
Then at a crossing on a diagram of the link, the colors $\{p|q|r\}$ 
are transformed to $\{p-\alpha|q-\alpha|r-\alpha\}$. 
We see $(p-\alpha)+(r-\alpha)=2(q-\alpha)$ from $p+r=2q$. 
Therefore the map $\gamma$ is a $\mathbb{Z}$-coloring such that \textit{Im}$(\gamma)=\{0,a_1,a_2,\cdots,a_n\}$ 
with $a_i >0$ $(i=1,2,\cdots,n)$.
\end{proof}

\begin{lemma}\label{lem1}
For a $\mathbb{Z}$-coloring $\gamma$ with $0=\min$\textit{Im}$(\gamma)$,
if an over arc at a crossing is colored by $0$, 
then the under arcs at the crossing are colored by $0$.
\end{lemma}

\begin{proof}
Suppose that a crossing has colors 
$\{a|b|c\}$ with $a\neq b\neq c$ for a $\mathbb{Z}$-coloring $\gamma$ with $0=\min$\textit{Im}$(\gamma)$.
Then we obtain $a>b>c$ or $c>b>a$. 
If we suppose that $b=0$, then it gives $0>c$ or $0>a$, 
contradicting to $0=\min$\textit{Im}$(\gamma)$.
\end{proof}

\begin{lemma}\label{lem2}
For a $\mathbb{Z}$-coloring $\gamma$ with $M=\max$ \textit{Im}$(\gamma)$,
if an over arc at a crossing is colored by $M$, 
then the under arcs at the crossing are colored by $M$.
\end{lemma}

\begin{proof}
Suppose that a crossing has colors 
$\{a|b|c\}$ with $a\neq b\neq c$ for a $\mathbb{Z}$-coloring $\gamma$ with $M=\max$\textit{Im}$(\gamma)$.
Then we obtain $a>b>c$ or $c>b>a$. 
If we suppose that $b=M$, then it gives $M<c$ or $M<a$, 
contradicting to $M=\max$\textit{Im}$(\gamma)$.
\end{proof}

%%%%%%%%%%%%%%%%%%%%%%%%%%%%%%%%%%%%%%%%%%%%%%%%%%%%%%%%%%%%%%%

\section{$\mathbb{Z}$-colorings by four colors}\label{sec:4colors}

In this section, we consider the case that a link can be colored by four colors. 
First, we can see that 
any splittable link $L$ is $\mathbb{Z}$-colorable and $mincol_\mathbb{Z}(L)=2$. 
On the other hand, the next holds for non-splittable links.

\begin{theorem}\label{thm1}
Let $L$ be a non-splittable $\mathbb{Z}$-colorable link. 
Then $mincol_\mathbb{Z}(L)\geq4$.
\end{theorem}

\begin{proof}
Let $\gamma$ be a non-trivial $\mathbb{Z}$-coloring for a $\mathbb{Z}$-colorable link $L$. 
We will show that if $mincol_\mathbb{Z} (L)\leq3$, 
then $L$ is splittable.

Since $\gamma$ is non-trivial, 
the cardinality of \textit{Im}$(\gamma)$ is greater than $1$. 

In the case 
that the cardinality of \textit{Im}$(\gamma)$ is $2$, 
we can assume that \textit{Im}$(\gamma)=\{0,a\}$ by Lemma \ref{lem0}. 
By Lemma \ref{lem2} 
a diagram of $L$ has no crossings with over arcs labeled by $a$ other than $\{a|a|a\}$. 
And it also has no crossings with over arcs labeled by $0$ other than $\{0|0|0\}$. 
Therefore the diagram is splittable, 
and we see that the link is splittable. 

In the case that 
the cardinality of \textit{Im}$(\gamma)$ is $3$, 
we can assume \textit{Im}$(\gamma)=\{0,a,b\}$ by Lemma \ref{lem0}. 
Let $b$ be the maximum of \textit{Im}$(\gamma)$. 
By Lemma \ref{lem2}, 
a diagram of the link has no crossings with over arcs labeled by $b$ other than $\{b|b|b\}$. 
If the diagram has no crossing colored by $\{0|a|b\}$, 
it has only crossings colored by $\{0|0|0\}$, 
$\{a|a|a\}$ or $\{b|b|b\}$. 
Then we obtain the link is splittable.
If the diagram has a vertex colored by $\{0|a|b\}$, 
we see $b=2a$.
From Lemma \ref{lem1} and \ref{lem2}, 
the diagram has no crossings labeled by $\{a|*|*\}$ 
with the exception of crossings labeled by $\{a|a|a\}$.
Then we see that the link is splittable.
\end{proof}

If a link is $\mathbb{Z}$-colorable with four colors, 
we can show the following.

\begin{theorem}\label{thm2}
If $mincol_\mathbb{Z}(L)=4$ 
for a $\mathbb{Z}$-colorable link $L$, 
then there exists a diagram $D$ of $L$ and a $\mathbb{Z}$-coloring $\gamma$ on $D$ 
such that \textit{Im}$(\gamma)=\{0,1,2,3\}$.
\end{theorem}

\begin{proof}
Suppose that there exists a $\mathbb{Z}$-coloring $\gamma'$ with four colors.
From Lemma \ref{lem0}, we can suppose \textit{Im}$(\gamma')=\{0,a,b,c\}$ with $a>0, b>0$ and $c>0$. 
We assume that 
$c$ is the maximum of \textit{Im}$(\gamma')$.
Since $L$ is a non-splittable link, 
we assume $0<a<b<c$. 
From Lemma \ref{lem1} and Lemma \ref{lem2}, $D$ has a crossing colored by $\{0|a|b\}$ or a crossing colored by $\{0|b|c\}$. 
If $D$ has both of the crossings, by the definition of $\mathbb{Z}$-coloring, we obtain $a-0=b-a$ and $b-0=c-b$. 
We see $b=2a$ and $c=2b$, and then we obtain \textit{Im}$(\gamma')=\{0,a,2a,4a\}$. 
Then the diagram has no crossings with colors $\{a|*|*\}$ other than $\{a|a|a\}$. 
This gives a contradiction to that $L$ is a non-splittable link.
If $D$ has only either of the crossings 
colored by $\{0|a|b\}$ or $\{0|b|c\}$, 
then $D$ has only 3 colors, 
contradicting to Theorem \ref{thm1}.
From the above, $D$ has either of the crossings and another crossing colored by $\{a|b|c\}$. 
If $D$ has a crossing colored by $\{0|b|c\}$ and a crossing colored by $\{a|b|c\}$, we obtain $b-0=c-b$ and $b-a=c-b$. 
We see $a=0$, and this is contradictory to the assumption. 
If $D$ has a crossing colored by $\{0|a|b\}$ and a crossing colored by $\{a|b|c\}$, we obtain $a-0=b-a$ and $b-a=c-b$. 
We see $b=2a$ and $c=3a$. 
Therefore, we obtain \textit{Im}($\gamma'$)$=\{0,a,2a,3a\}$. 
By dividing by $a$, 
we have a $\mathbb{Z}$-coloring with colors $\{0,1,2,3\}$ as desired.
\end{proof}

Among links of crossing numbers at most $9$, 
there are $5$ links with zero determinant. 
For the $\mathbb{Z}$-colorable links, 
the colorings on the diagrams in \cite{linkinfo} are quite distinctive. 
See figures in Section 6. 
In the next section, we consider such distinctive colorings. 

%%%%%%%%%%%%%%%%%%%%%%%%%%%%%%%%%%%%%%%%%%%%%%%%%%%%%%%%%%%%%%%%%%%%%%%%

\section{Reduction of colors}\label{sec:simple}

In this section, 
we focus on the \lq\lq simplest\rq\rq  \ 
$\mathbb{Z}$-coloring 
found for the links with at most 9 crossings. 
Based on such examples, 
we introduce the following notion. 

\begin{definition}
Let $L$ be a non-trivial $\mathbb{Z}$-colorable link, 
and $\gamma$ a $\mathbb{Z}$-coloring on a diagram $D$ of $L$. 
Suppose that there exists a natural number $d$ such that, at all the crossings in $D$, 
the differences between the colors of the over arcs and the under arcs are $d$ or $0$.
Then we call $\gamma$ a \textit{simple} $\mathbb{Z}$-coloring.
\end{definition}

For example, a pretzel knot $P(n,-n,n,-n,\cdots,n,-n)$ with integer $n$ admits a simple $\mathbb{Z}$-coloring. 
See Figure \ref{pretzelex}. 

The next is our first main result in this paper.

\begin{theorem}\label{simplethm}
Let $L$ be a non-splittable $\mathbb{Z}$-colorable link.
If there exists a simple $\mathbb{Z}$-coloring on a diagram of $L$,
then $mincol_\mathbb{Z}(L)=4$.
\end{theorem}

Throughout the rest of the paper, 
we depict a crossing colored by $\{a|a|a\}$ with an integer $a$ as shown in Figure \ref{aaa}.

\begin{figure}[H]
\begin{center}
\includegraphics[width=3.5cm,clip,bb=0 0 185 186]{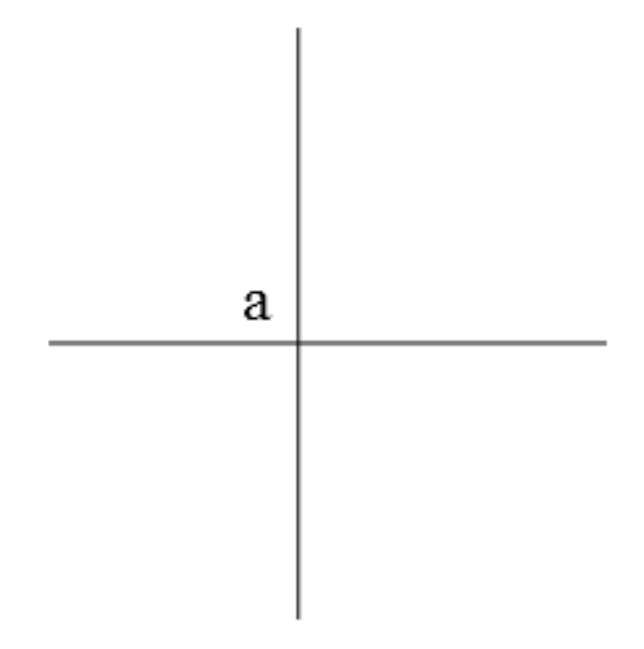}
\caption{}\label{aaa}
\end{center}
\end{figure}

\begin{proof}[Proof of Theorem \ref{simplethm}]
Let $\gamma$ be a simple $\mathbb{Z}$-coloring on a diagram of a non-splittable $\mathbb{Z}$-colorable link $L$. 
Then there exists a natural number $d$ such that, 
at all the crossing in $D$, 
the differences between the colors of the over arcs 
and the under arcs are $d$ or $0$. 
From Lemma \ref{lem1}, 
we regard that $0$ is the minimum of \textit{Im}$(\gamma)$.
For the maximum $M$ of \textit{Im}$(\gamma)$, 
by Lemma \ref{lem2}, the diagram has only crossings colored by $\{M|M|M\}$ or $\{M|M-d|M-2d\}$. 
First we delete a crossing colored by $\{M|M|M\}$ as shown in Figure \ref{reductionM}.

\begin{figure}[H]
\begin{center}
\includegraphics[width=11cm,clip,bb=0 0 740 273]{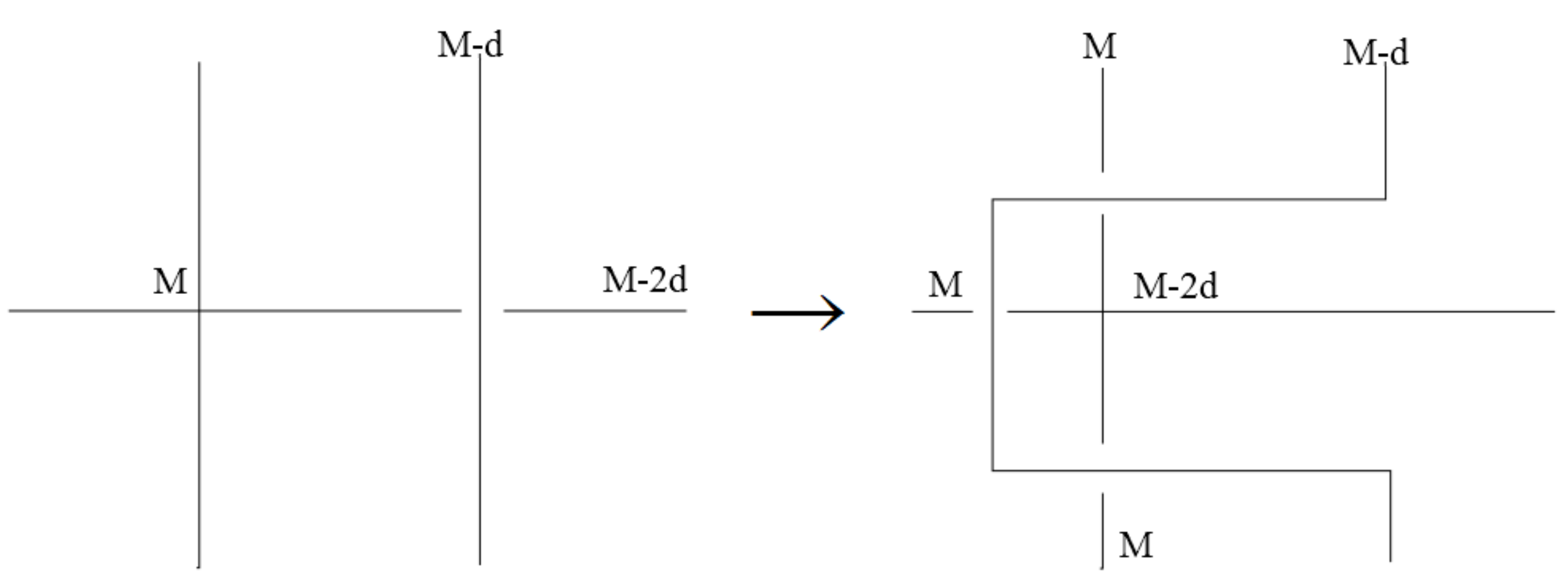}
%\caption{delete $\{M|M|M\}$}\label{reductionM}
\caption{}\label{reductionM}
\end{center}
\end{figure}

Next, we transform the diagram inductively near the crossings colored by $\{M|M-d|M-2d\}$ as shown in Figure \ref{reduction}. 

\begin{figure}[H]
\begin{center}
\includegraphics[width=4.6cm,clip,bb=0 0 333 225]{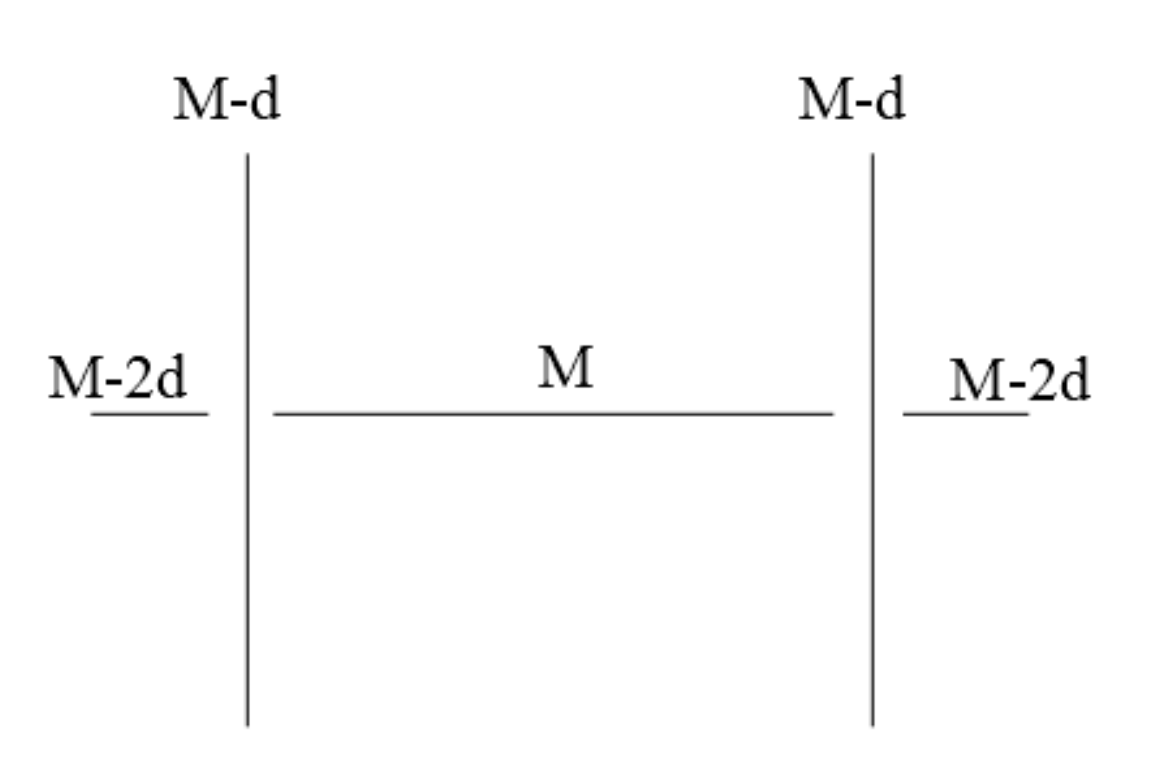}
\includegraphics[height=3cm,clip,bb=0 0 80 365]{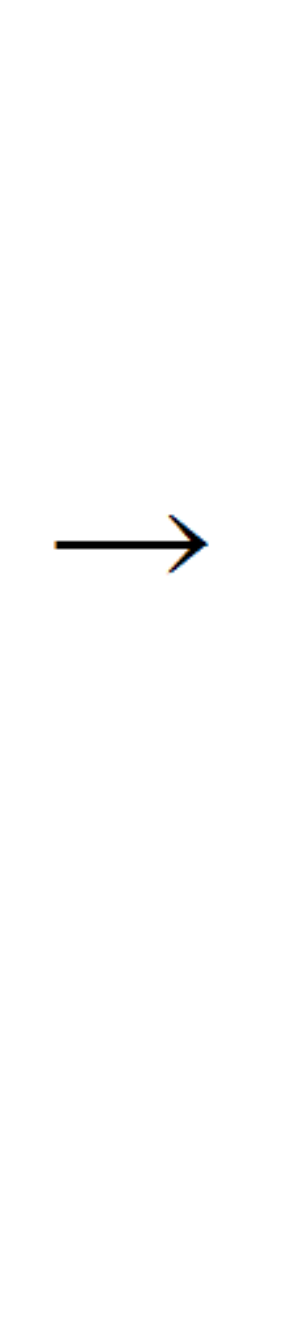}
\includegraphics[width=7cm,clip,bb=0 0 586 340]{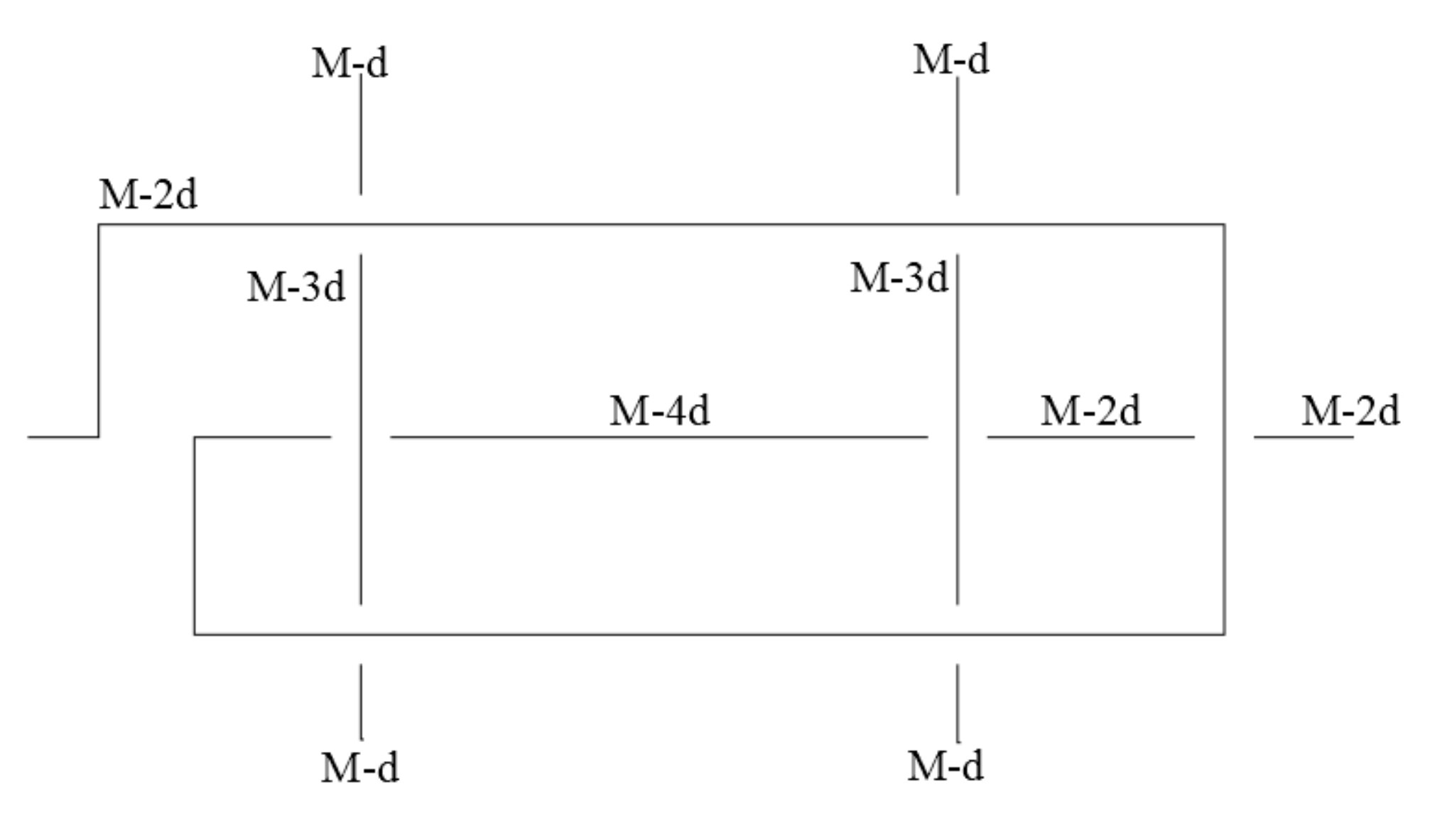}
%\caption{reduction of $M$}\label{reduction}
\caption{}\label{reduction}
\end{center}
\end{figure}

Then we obtain a simple $\mathbb{Z}$-coloring such that the maximum of colors is less than that for the previous one.

Note that after changing the diagram, the common difference is still $d$ or $0$. 
If the maximum of colors for the obtained $\mathbb{Z}$-coloring is at least $4d$, 
then we perform the transformation repeatedly. 
After all, we obtain a $\mathbb{Z}$-coloring with $\{0,d,2d,3d\}$. 
Therefore $mincol_\mathbb{Z}(L)$ is at most $4$.
From Lemma \ref{lem1}, we see $mincol_\mathbb{Z}(L)=4$.
\end{proof}

%From Theorem \ref{simplethm}, 
%the diagram with a simple $\mathbb{Z}$-coloring 
%as shown in Figure \ref{pretzel} can be transformed 
%to Figure \ref{pretzel2} without the previous maximum $n$.
%
%\begin{figure}[H]
%\begin{center}
%\includegraphics[width=12cm,clip]{pretzel2.pdf}
%\caption{$P(n,-n,n,-n,\cdots,n,-n)$}\label{pretzel2}
%\end{center}
%\end{figure}

%%%%%%%%%%%%%%%%%%%%%%%%%%%%%%%%%%%%%%%%%%%%%%%%%%%

\section{$\mathbb{Z}$-coloring by five colors}\label{sec:five}

In the previous section, we investigated the case that 
a diagram with a simple $\mathbb{Z}$-coloring. 
However, there are many diagrams of $\mathbb{Z}$-colorable links without simple $\mathbb{Z}$-colorings. 
See Figure \ref{L10n32ex} for example.
In this section, 
we focus on the case that a diagram is colored by five colors. 
We first see the following in this case.

\begin{figure}[H]
\begin{center}
\includegraphics[width=8cm,clip,bb=0 0 453 414]{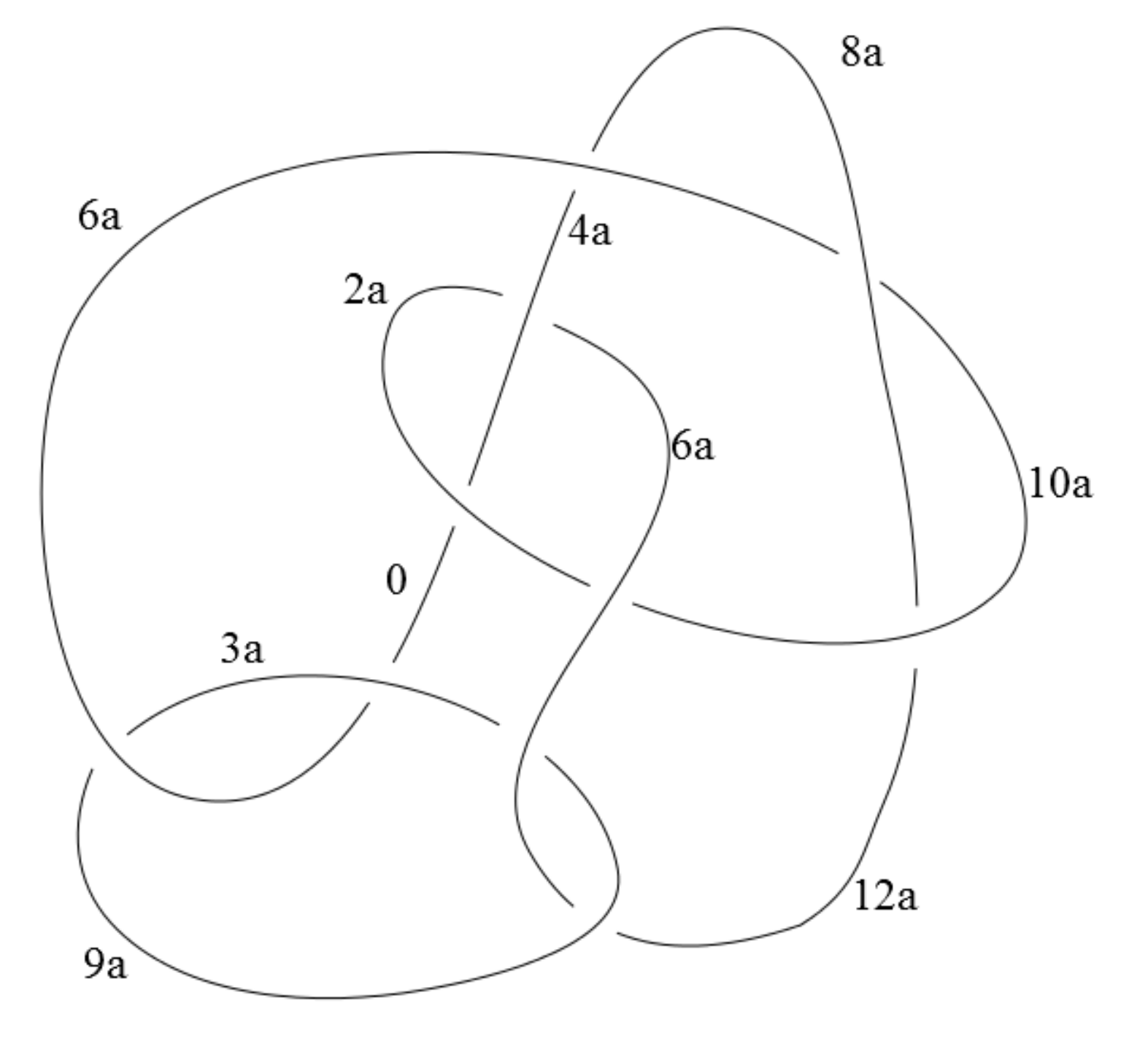}
\caption{Non-simple $\mathbb{Z}$-coloring for $L10n32$}\label{L10n32ex}
\end{center}
\end{figure}

\begin{theorem}\label{thm41}
If a non-splittable $\mathbb{Z}$-colorable link admits a $\mathbb{Z}$-coloring with five colors,
then there exists a $\mathbb{Z}$-coloring $\gamma$ for $L$ 
with {\it Im}$(\gamma)=\{0,1,2,3,4\}$, $\{0,1,2,3,5\}$, $\{0,1,2,3,6\}$, $\{0,1,2,4,7\}$, $\{0,2,3,4,5\}$, $\{0,3,4,5,6\}$ or $\{0,3,5,6,7\}$.
\end{theorem}

To prove the theorem above, we use the following notion, 
which is originally introduced for Fox coloring in \cite{NakamuraNakanishiSatoh2}.

\begin{definition}
For a $\mathbb{Z}$-coloring $\gamma$ 
of a diagram of a $\mathbb{Z}$-colorable link, 
let $c_1,c_2,\cdots,c_l$ be the distinct colors 
on the arcs $x_1,x_2,\cdots,x_n$. 
We define {\it the palette graph} $G$ associated with $\gamma$ 
as follows;\\
(1) The vertices of $G$ are the colors $c_1,c_2,\cdots,c_l$.\\
(2) Two vertices $c_i$ and $c_j$ are connected by an edge 
if and only if 
there is a crossing of the diagram whose under arcs are colored 
by $c_i$ and $c_j$.\\
We label the edge in (2) by $(c_i+c_j)/2$, 
which is coincident with one of $c_1,c_2,\cdots,c_l$ 
assigned to the over arc.
\end{definition}

For a $\mathbb{Z}$-coloring on a diagram of a link with at least 2 components, 
the palette graph can be disconnected. 
For example, Figure \ref{l10n36pallet} illustrate the palette graph for the link $L10n36$ with a $\mathbb{Z}$-coloring.

\begin{figure}[H]
\begin{center}
\includegraphics[width=12cm,clip,bb=0 0 757 366]{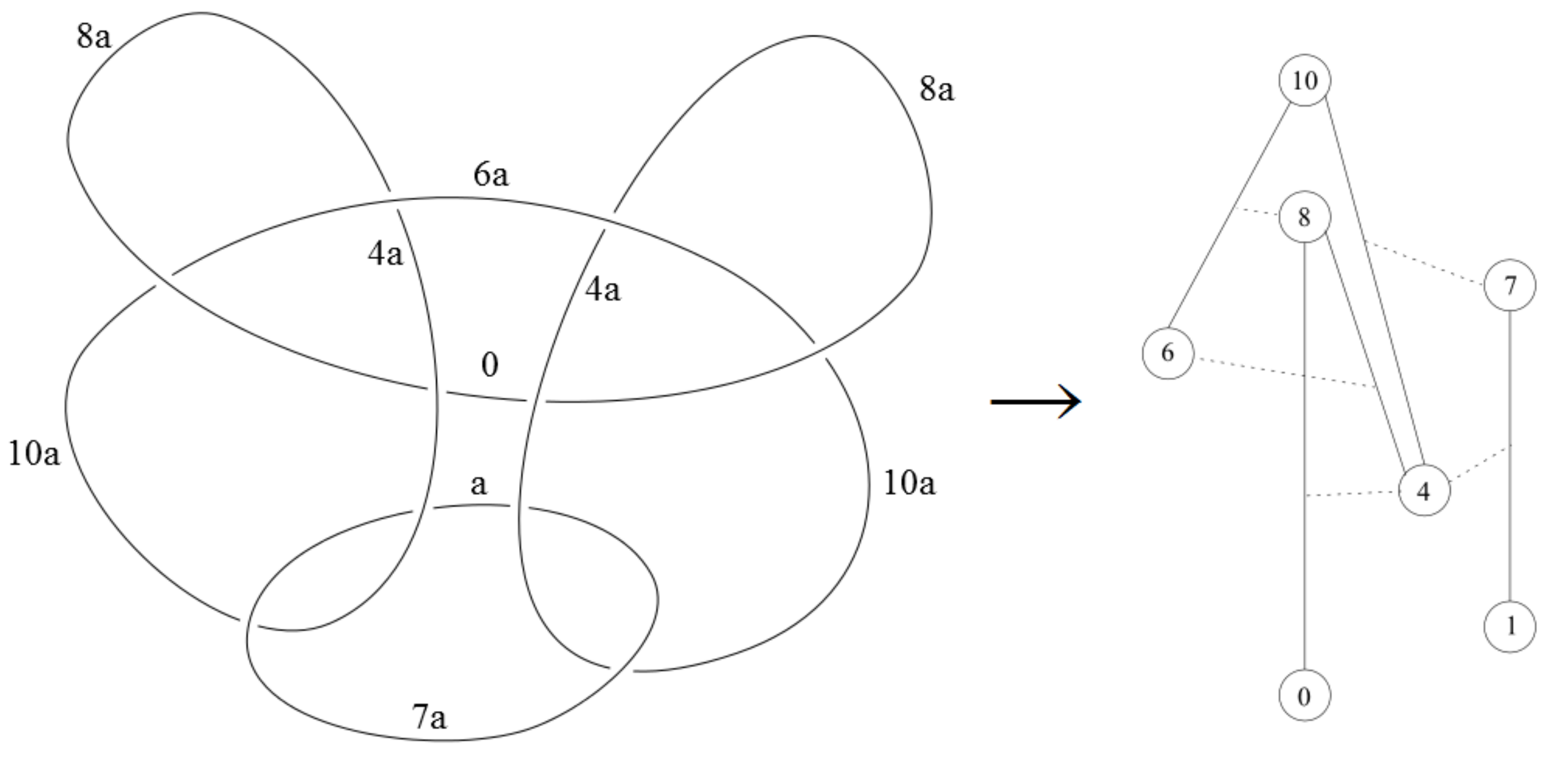}
\caption{Palette graph for L10n36}\label{l10n36pallet}
\end{center}
\end{figure}

\begin{proof}[Proof of Theorem \ref{thm41}]
We suppose that a non-splittable $\mathbb{Z}$-colorable link admits a $\mathbb{Z}$-coloring $\gamma$ with five colors.
By Lemma \ref{lem1}, without loss of generality, we assume that the image of $\gamma$ is $\{0,a,b,c,d\}$ with $0\neq a\neq b\neq c\neq d$. 
There are even numbers and odd numbers in the image of $\gamma$. 
A vertex labeled by an even number and a vertex labeled by an odd number are not connected by definition of the palette graph. 
Since the link is non-splittable, there are at least 2 vertices with even numbers and at least 2 vertices with odd numbers. 
Furthermore any vertex is connected to another vertex. 
Therefore the palette graph of the diagram is a $2$-component graph. 

In the case that the vertices colored by $0,a$ are connected and the vertices colored by $b,c,d$ are connected 
as shown in Figure \ref{0a}, we obtain the cases that {\it Im}$(\gamma)=\{0,1,2,3,5\}$ and $\{0,3,5,6,7\}$.

\begin{figure}[H]
\begin{center}
\includegraphics[width=8cm,clip,bb=0 0 325 185]{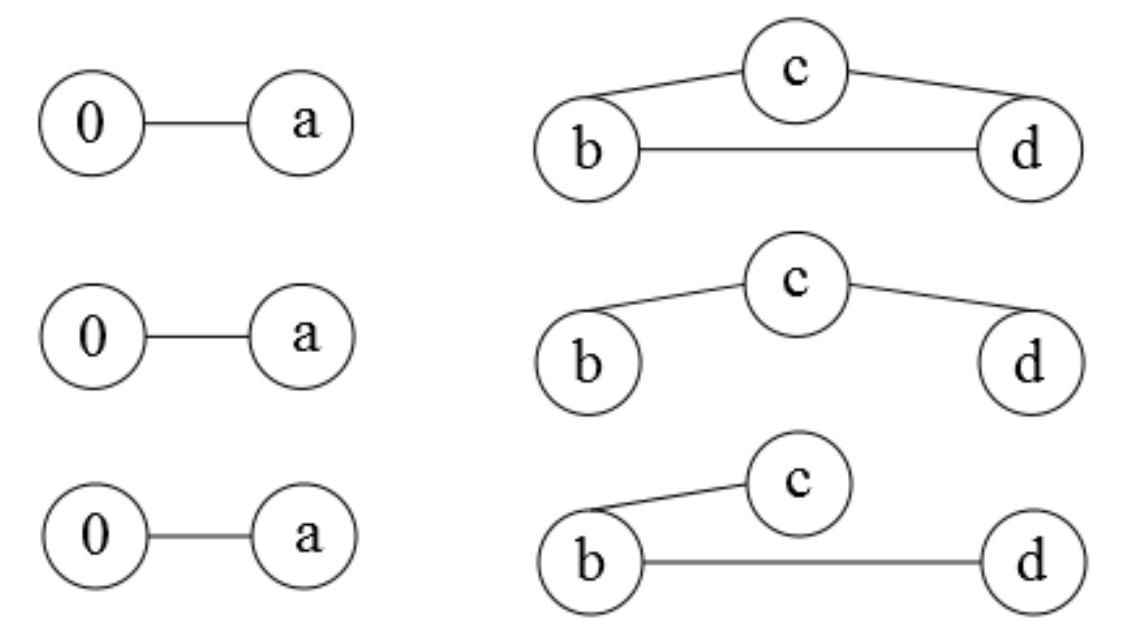}
\caption{ }\label{0a}
\end{center}
\end{figure}

In the case the vertices colored by $0,a,b$ 
are connected and the vertices colored by $c,d$ 
are connected as shown in Figure \ref{0ab}, 
we obtain the cases that {\it Im}$(\gamma)=\{0,1,2,3,4\}$, $\{0,1,2,3,6\}$, $\{0,2,3,4,5\}$, $\{0,3,4,5,6\}$ and $\{0,1,2,4,7\}$.

\begin{figure}[H]
\begin{center}
\includegraphics[width=8cm,clip,bb=0 0 327 188]{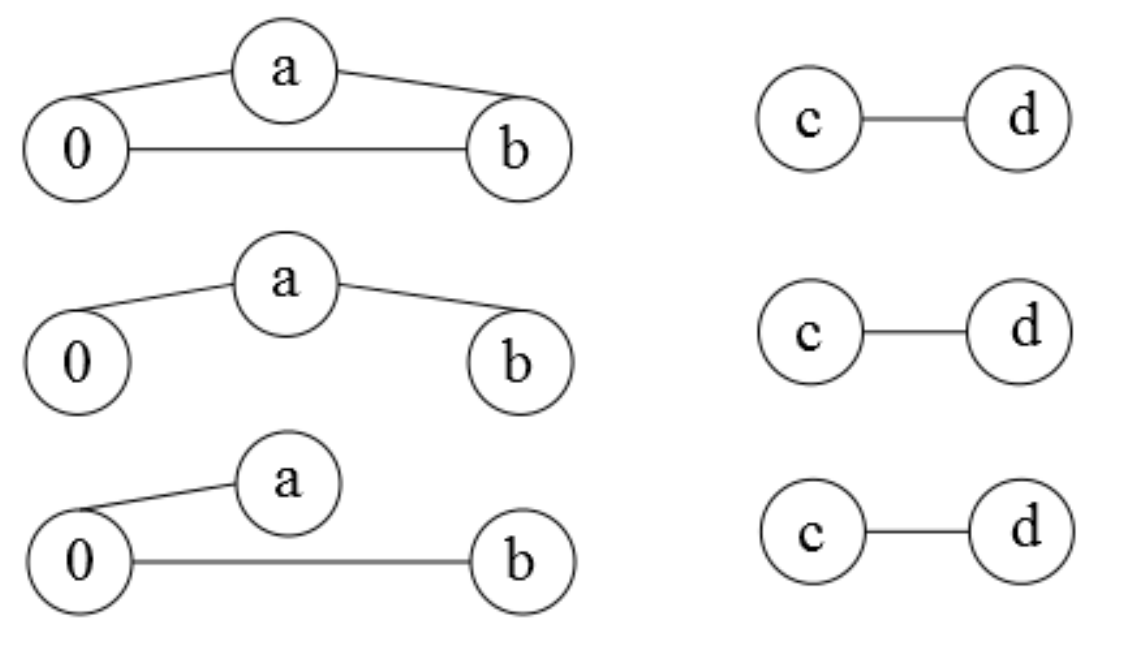}
\caption{ }\label{0ab}
\end{center}
\end{figure}

\end{proof}

\begin{remark}\label{rem01234}
There exist four palette graphs 
associated with a $\mathbb{Z}$-coloring $\gamma$ 
with {\it Im}$(\gamma)=\{0,1,2,3,4\}$ 
as shown in Figure \ref{01234}. 

\begin{figure}[H]
\begin{center}
\includegraphics[width=10cm,clip,bb=0 0 456 183]{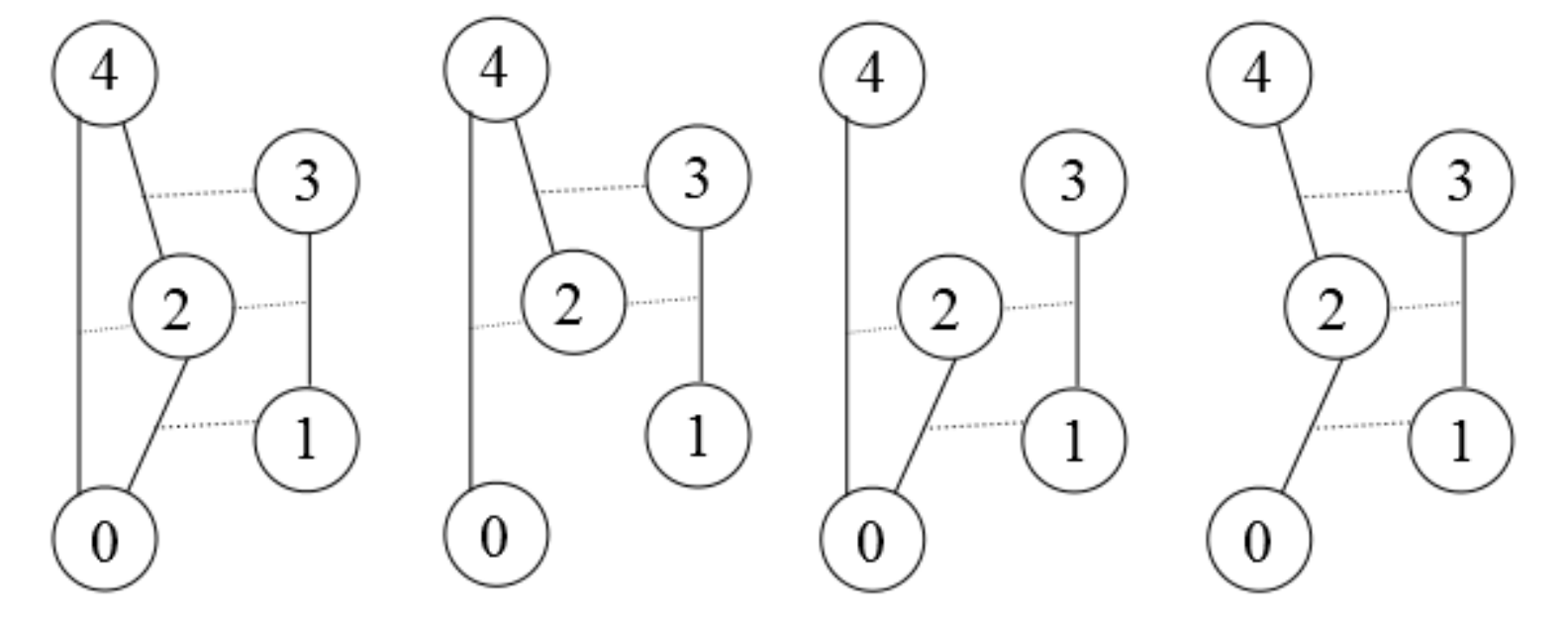}
\caption{Palette graphs with the image $\{0,1,2,3,4\}$}\label{01234}
\end{center}
\end{figure}
\end{remark}

The following is our second main theorem.

\begin{theorem}\label{thm42}
If a non-splittable link $L$ admits a $\mathbb{Z}$-coloring $\gamma$ 
with five colors,  
then $mincol_\mathbb{Z}(L)=4$.
\end{theorem}

\begin{remark}
In the case that \textit{Im}$(\gamma)=\{0,1,2,3,4\}$, 
it cannot be assumed that $\gamma$ is a simple $\mathbb{Z}$-coloring, as seen in Remark \ref{rem01234}.
\end{remark}

\begin{proof}[Proof of Theorem \ref{thm42}]
Let $D$ be a diagram of a non-splittable link $L$ 
and $\gamma$ a $\mathbb{Z}$-coloring on $D$.

We first consider the case that \textit{Im}$(\gamma)=\{0,1,2,3,4\}$. 

We start with deleting the crossings colored by $\{4|4|4\}$ 
as illustrated in Figure \ref{Red01}. 
In the figure, $3(2)$ and $2(0)$ indicate two cases of the colors of the corresponding arcs, 
that is, when one is $3$ (resp. $2$), then the other is $2$ (resp. $0$). 

\begin{figure}[H]
\begin{center}
\includegraphics[width=11cm,clip,bb=0 0 740 273]{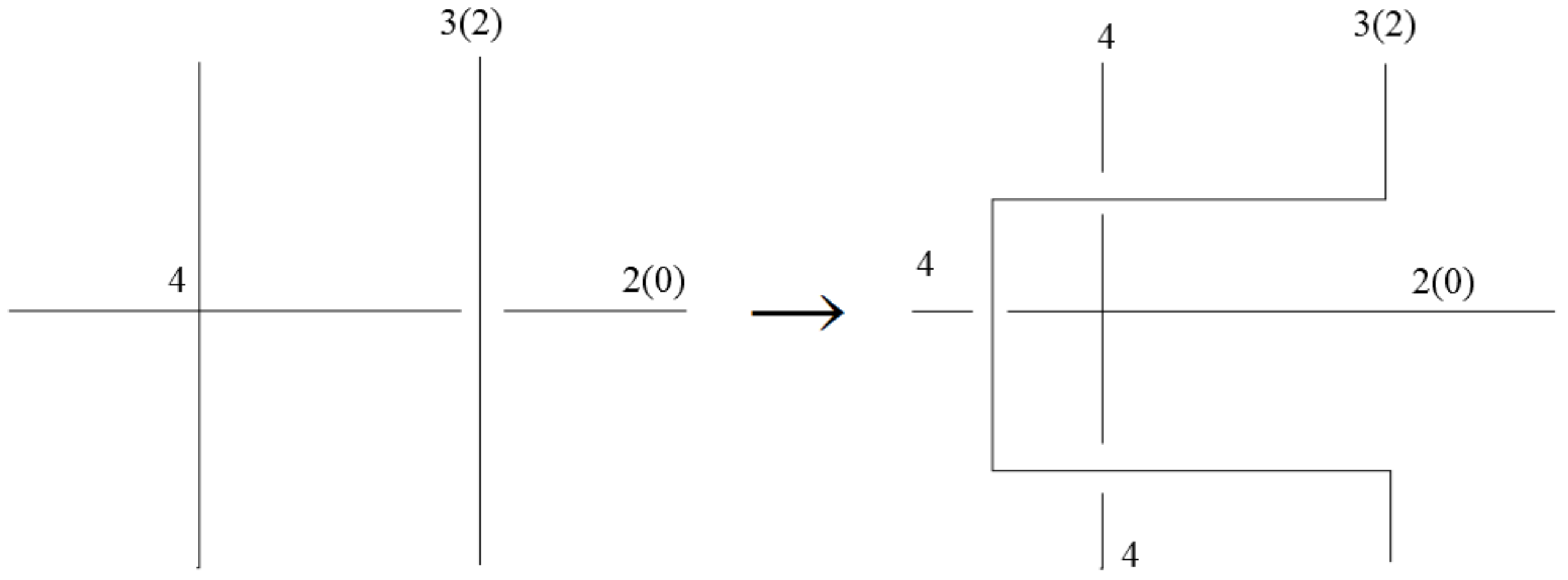}
\caption{}\label{Red01}
\end{center}
\end{figure}

Now, since $4$ is the maximum in {\it Im}$(\gamma)$, the arc colored by $4$ must be an under arc at any crossing. 
Then we will delete the crossings colored by $\{4|*|*\}$ as follows. 

In the case that the diagram has colored crossings shown in Figure \ref{Red-a}, we transform the diagram and the coloring as shown in Figure \ref{Red-b} or \ref{Red-c}.

\begin{figure}[H]
\begin{center}
\includegraphics[width=5cm,clip,bb=0 0 388 260]{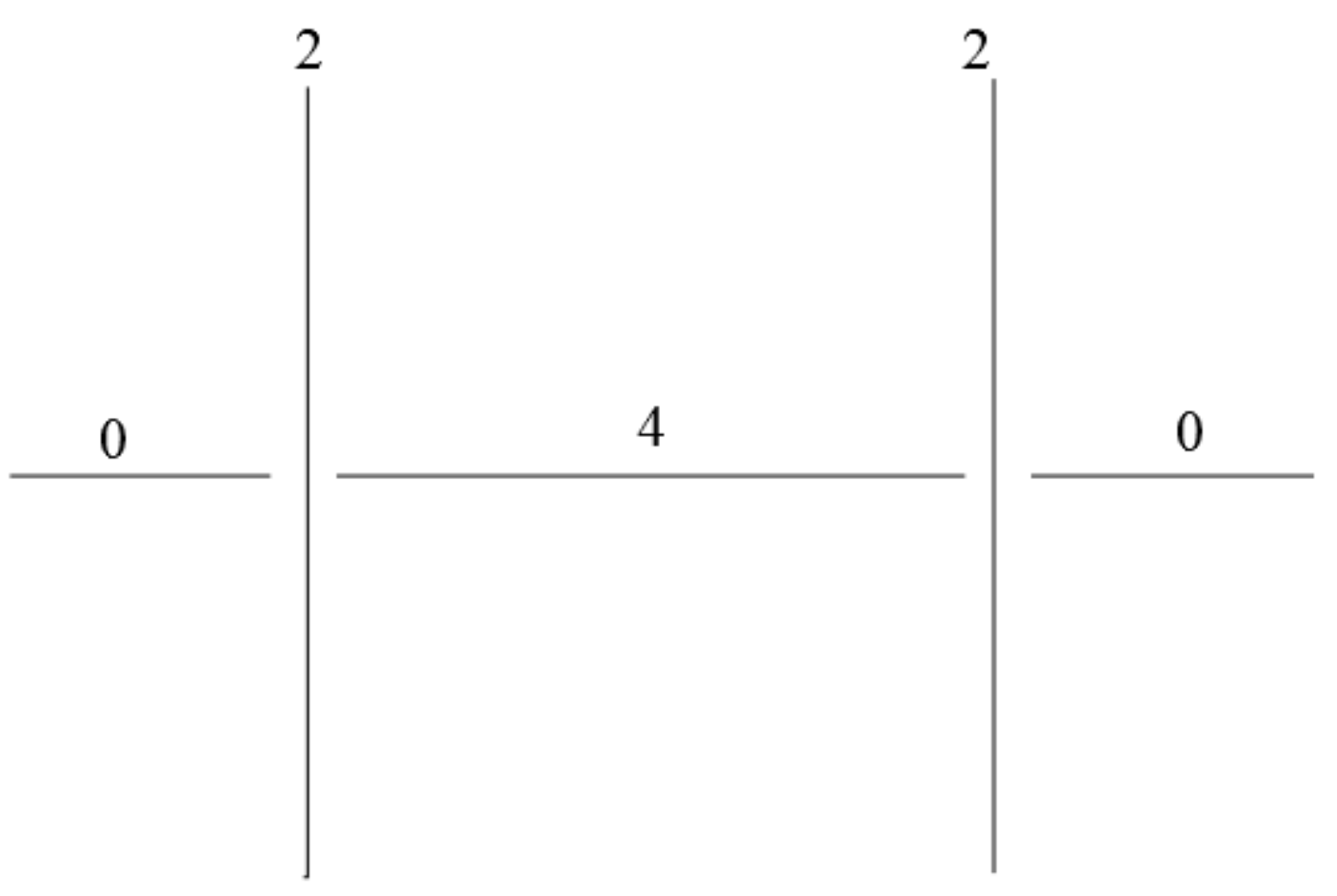}
\caption{}\label{Red-a}
\end{center}
\end{figure}

\begin{figure}[H]
\begin{center}
\includegraphics[width=9cm,clip,bb=0 0 487 242]{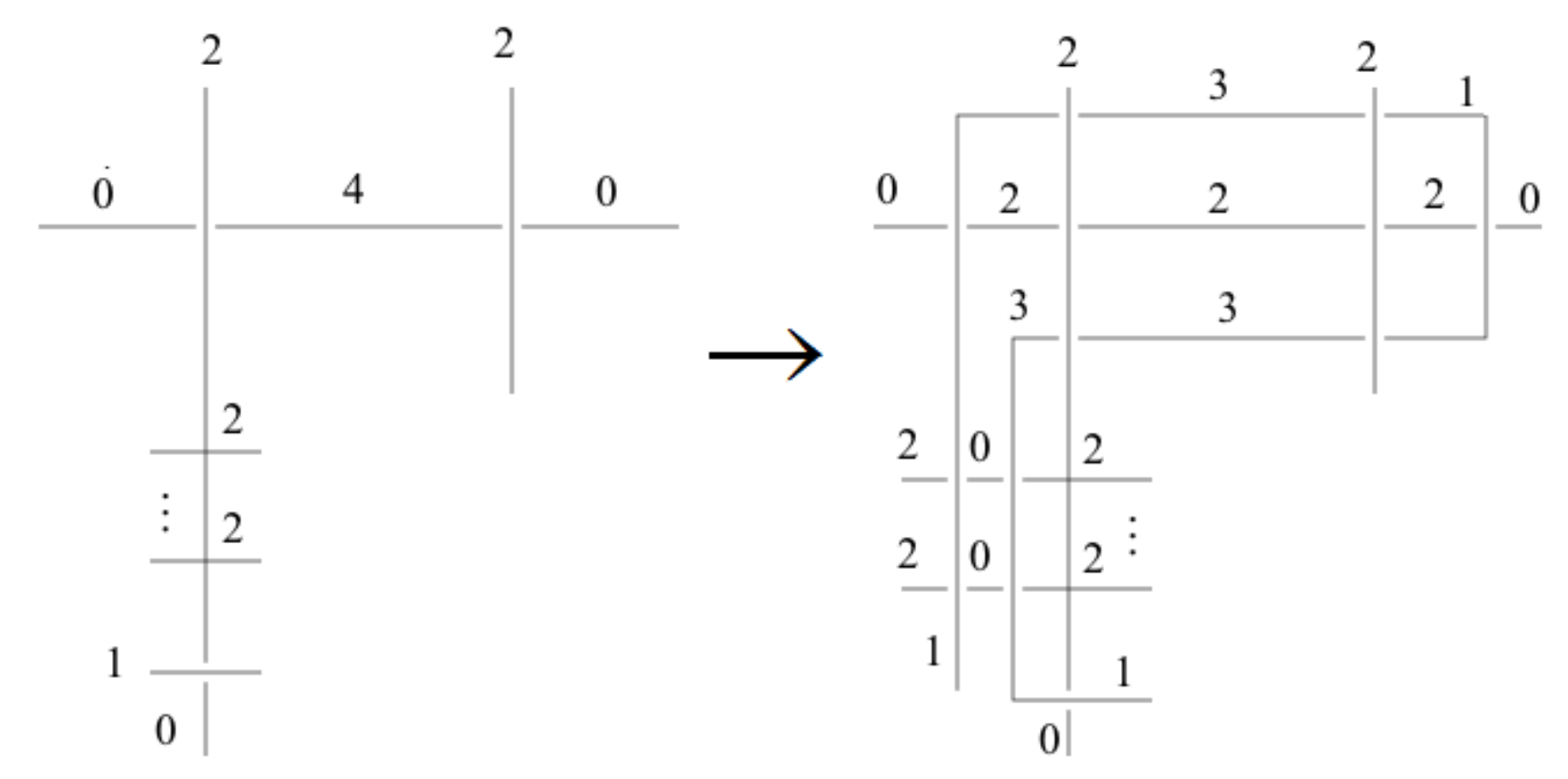}
\caption{}\label{Red-b}
\end{center}
\end{figure}

\begin{figure}[H]
\begin{center}
\includegraphics[width=9cm,clip,bb=0 0 441 242]{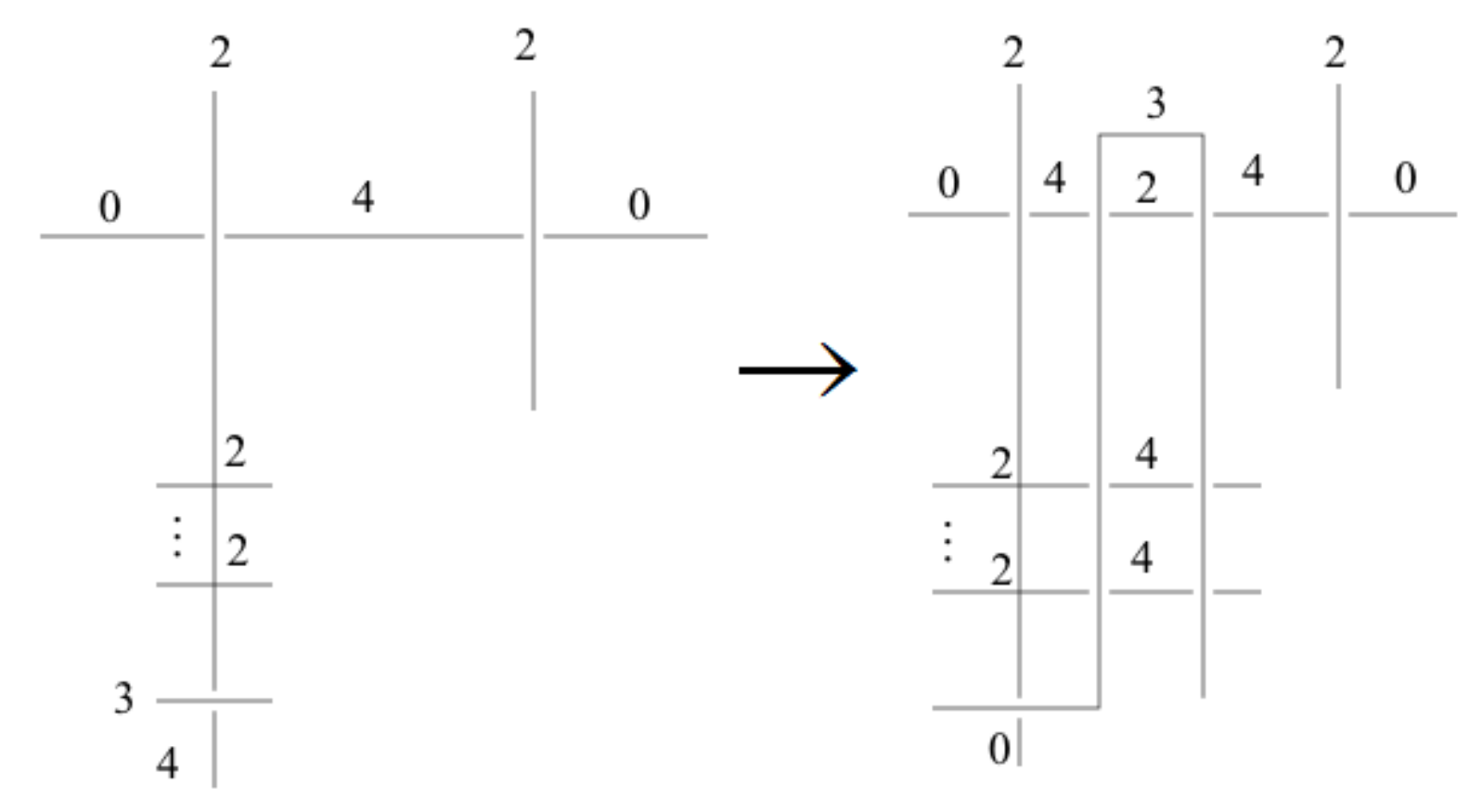}
\caption{}\label{Red-c}
\end{center}
\end{figure}

In the case that a crossing colored by $\{4|3|2\}$ exists, 
we transform the diagram and the coloring 
as illustrated in Figure \ref{Red02} and \ref{Red03}.

\begin{figure}[H]
\begin{center}
\includegraphics[width=11cm,clip,bb=0 0 765 294]{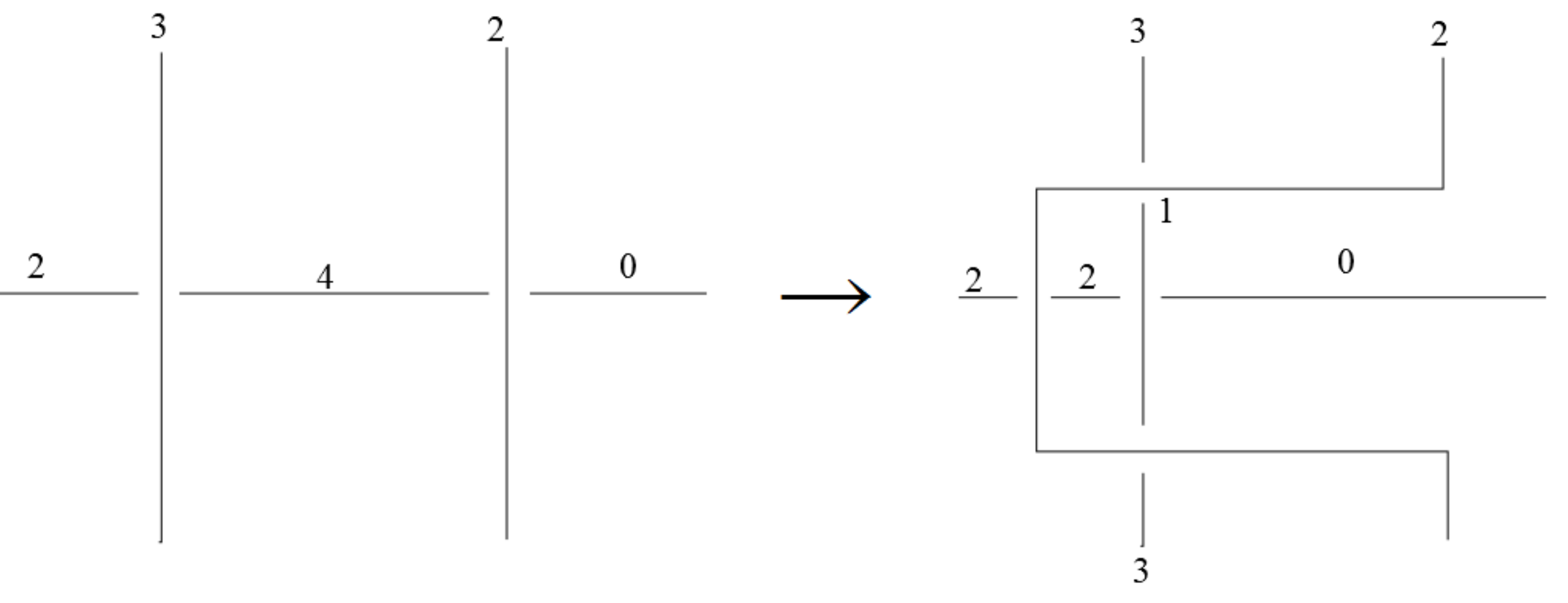}
\caption{}\label{Red02}
%\end{center}
%\end{figure}
%
%\begin{figure}[H]
%\begin{center}
\includegraphics[width=13cm,clip,bb=0 0 846 272]{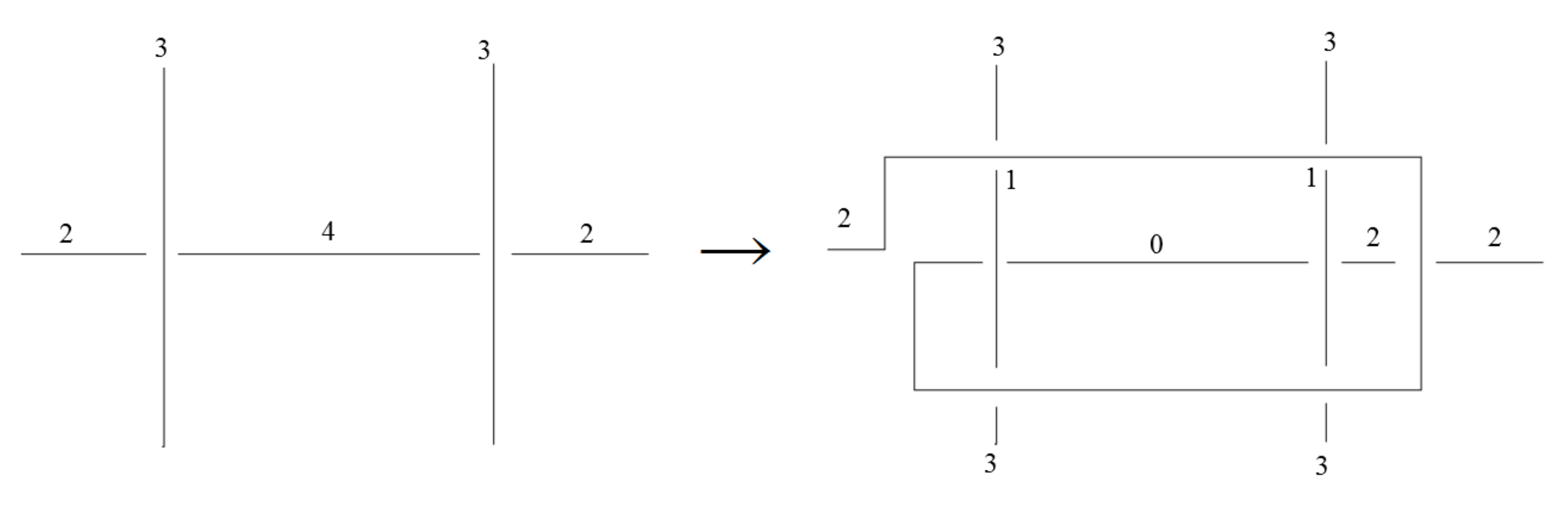}
\caption{}\label{Red03}
\end{center}
\end{figure}

Consequently we obtain a simple $\mathbb{Z}$-coloring. 
From Theorem \ref{simplethm}, 
we see that $mincol_\mathbb{Z}(L)=4$.

%{0,1,2,3,5}

In the case that {\it Im}$(\gamma)=\{0,1,2,3,5\}$, first we delete crossings colored by $\{5|5|5\}$ as shown in Figure \ref{Red12}.

\begin{figure}[H]
\begin{center}
\includegraphics[width=12cm,clip,bb=0 0 712 270]{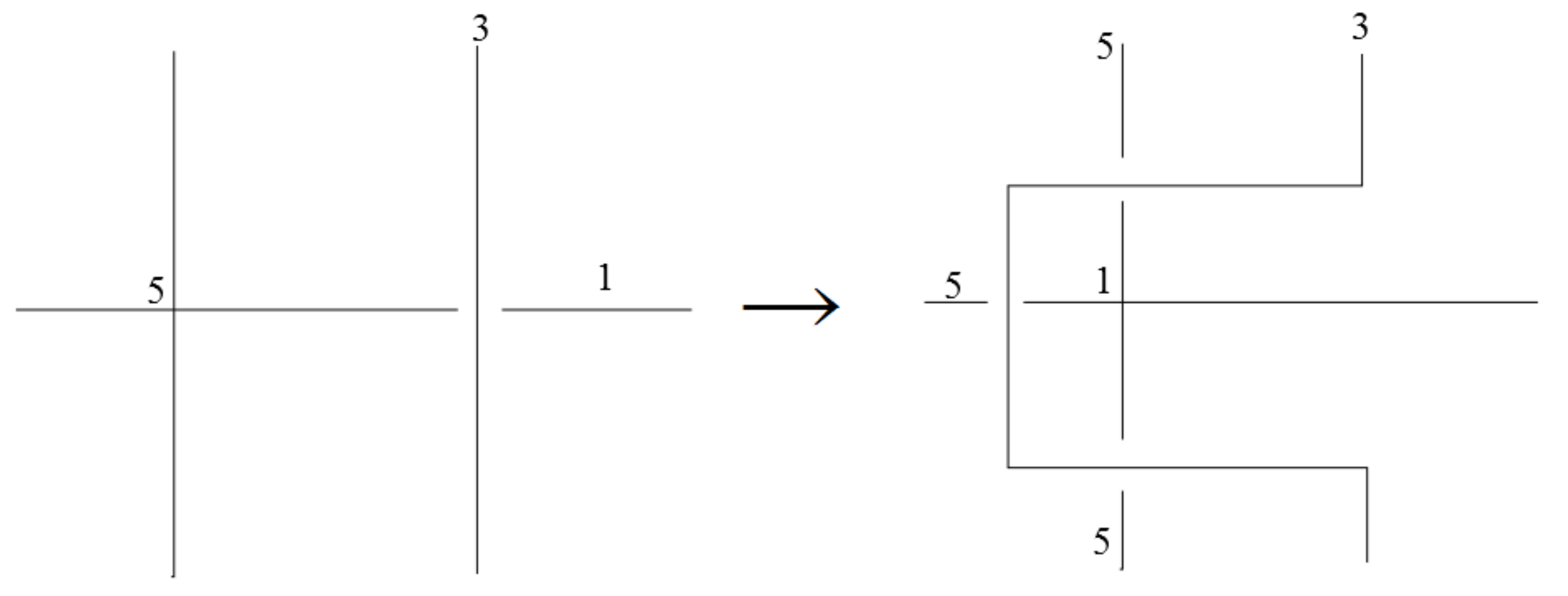}
\caption{}\label{Red12}
\end{center}
\end{figure}

Next we focus on a crossing colored by $\{1|3|5\}$. 
On the extension of the arc colored by $3$, 
a crossing colored by $\{3|2|1\}$ exists. 
Then we transform the diagram as shown in Figure \ref{Red13}.

\begin{figure}[H]
\begin{center}
\includegraphics[width=13cm,clip,bb=0 0 847 406]{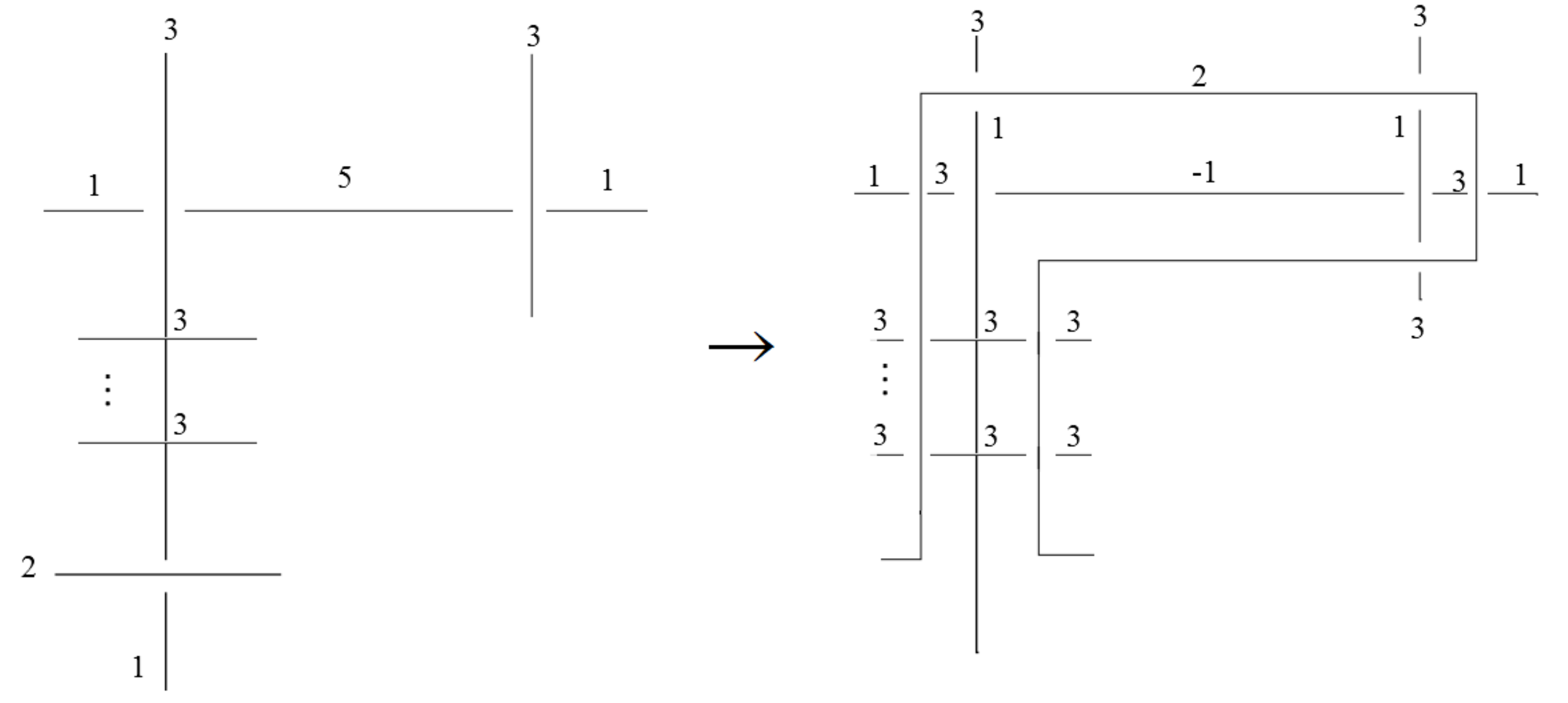}
\caption{}\label{Red13}
\end{center}
\end{figure}

Then we see that the diagram is colored by $\{-1,0,1,2,3\}$. 
We add $1$ to all the colors, and then we obtain the colors $\{0,1,2,3,4\}$. 
Then we reduce this case to the case of {\it Im}$(\gamma)=\{0,1,2,3,4\}$, and we see $mincol_\mathbb{Z}(L)=4$.\\

In the case that {\it Im}$(\gamma)=\{0,2,3,4,5\}$, 
we add $-5$ and multiply $-1$ to all colors for $\gamma$. 
Then this case is reduced to the case of {\it Im}$(\gamma)=\{0,1,2,3,5\}$.\\

In the case that {\it Im}$(\gamma)=\{0,3,4,5,6\}$, 
the palette graph associated with $\gamma$ is shown in Figure \ref{palette03456}.

\begin{figure}[H]
\begin{center}
\includegraphics[width=3.5cm,clip,bb=0 0 225 333]{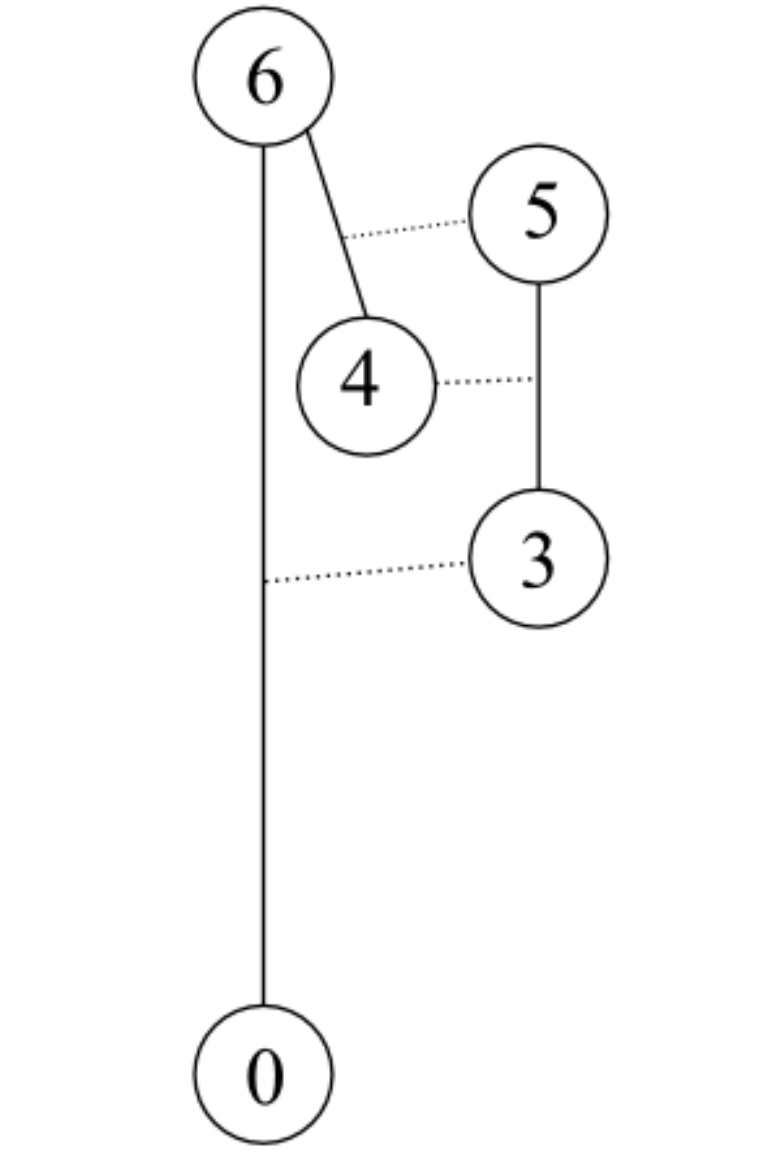}
\caption{}\label{palette03456}
\end{center}
\end{figure}

We focus on an arc colored by $0$. 
First we delete a crossing colored by $\{0|0|0\}$ as shown in Figure \ref{Red14}.

\begin{figure}[H]
\begin{center}
\includegraphics[width=12cm,clip,bb=0 0 688 263]{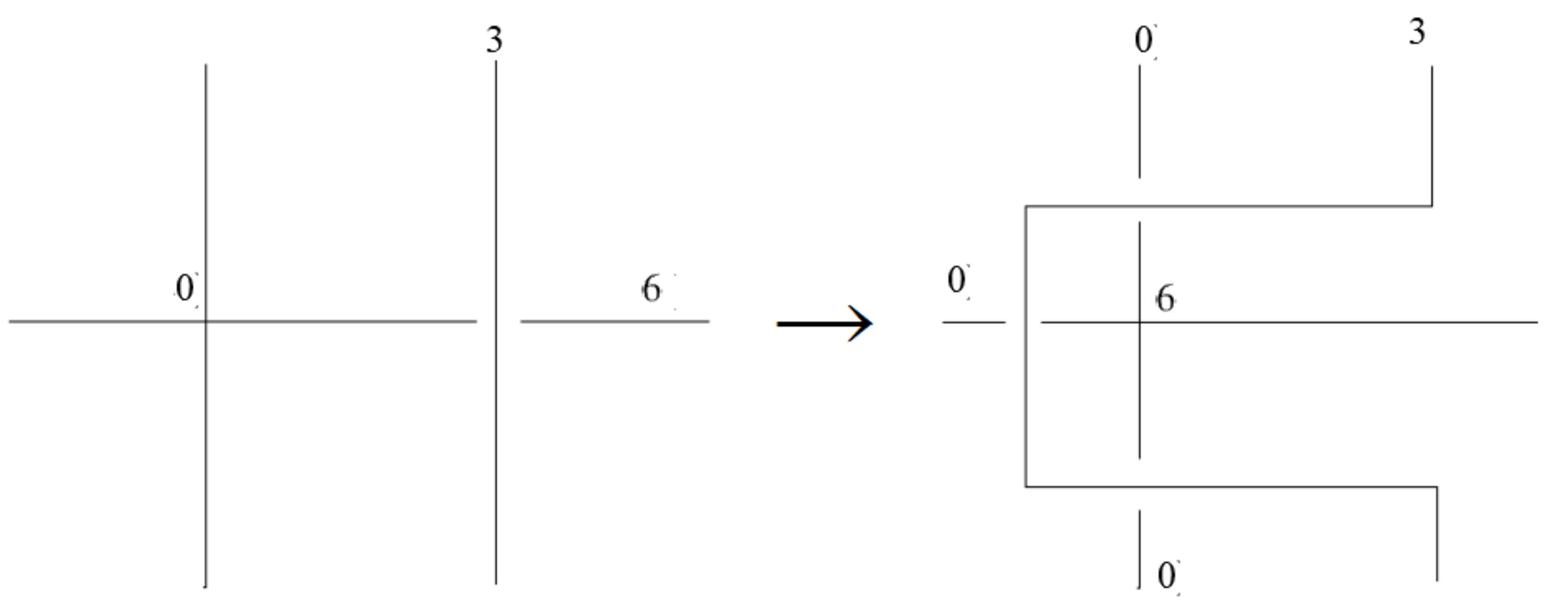}
\caption{}\label{Red14}
\end{center}
\end{figure}

Next we delete a crossing colored by $\{0|3|6\}$ as shown in Figure \ref{Red15}.

\begin{figure}[H]
\begin{center}
\includegraphics[width=13cm,clip,bb=0 0 566 273]{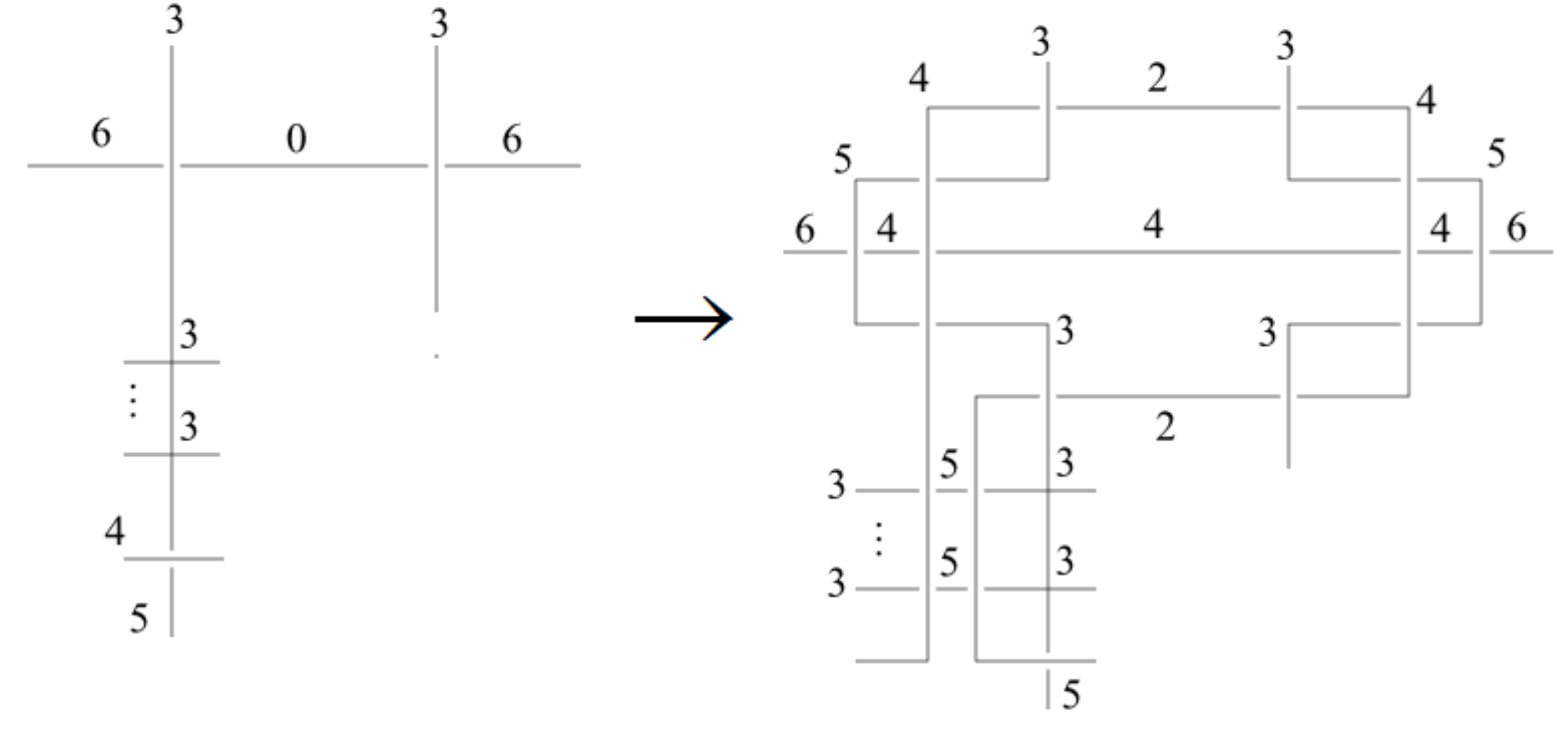}
\caption{}\label{Red15}
\end{center}
\end{figure}

Then we obtain another $\mathbb{Z}$-coloring $\gamma'$ such that {\it Im}$(\gamma')=\{2,3,4,5,6\}$. 
We add $-2$ to all the colors for $\gamma'$. 
Then we regard that this case is equivalence to the case {\it Im}$(\gamma)=\{0,1,2,3,4\}$.

In the case that {\it Im}$(\gamma)=\{0,1,2,3,6\}$, 
we add $-6$ and multiply $-1$ to all colors for $\gamma$. 
Then this case is reduced to the case of {\it Im}$(\gamma)=\{0,3,4,5,6\}$.
\bigskip

%%%%% 01247 %%%%

In the case that {\it Im}$(\gamma)=\{0,1,2,4,7\}$, 
the palette graph associated with $\gamma$ is shown in Figure \ref{palette01247a}.

\begin{figure}[H]
\begin{center}
\includegraphics[width=3.5cm,clip,bb=0 0 270 382]{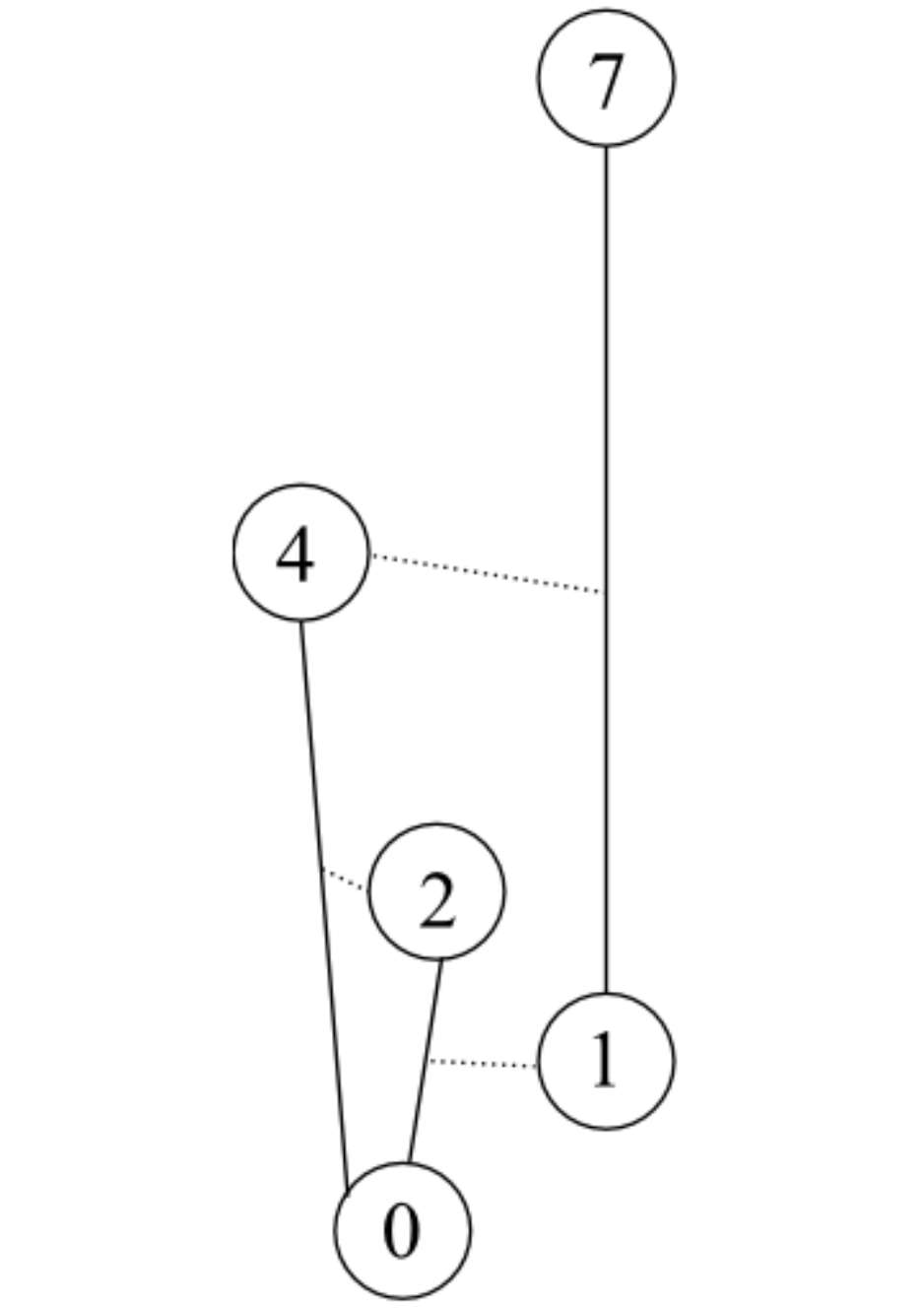}
\caption{}\label{palette01247a}
\end{center}
\end{figure}

First we delete a crossing colored by $\{0|0|0\}$. 
We transform the diagram and the coloring depicted in Figure \ref{Red17a}. 
In the figure, $1(2)$ and $2(4)$ indicate two cases of the colors of the corresponding arcs, that is, when one is 1 (resp. 2), then the other is 2 (resp. 4).

\begin{figure}[H]
\begin{center}
\includegraphics[width=14cm,clip,bb=0 0 740 273]{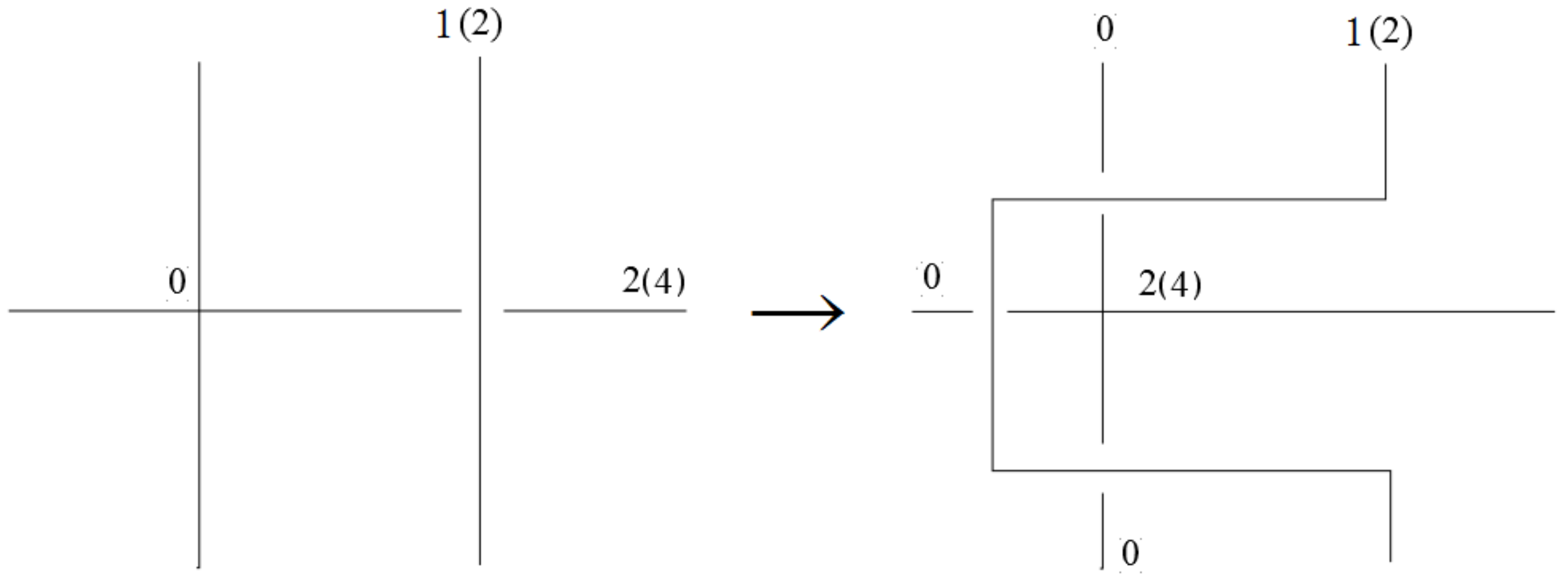}
\caption{}\label{Red17a}
\end{center}
\end{figure}

Next we delete a crossing colored by $\{0|2|4\}$. We transform the diagram and the coloring as follows.

\medskip

\noindent
[1] In the case that, on the extension of the arc colored by $0$, a crossing colored by $\{0|2|4\}$ exists, we transform the diagram as shown in Figure \ref{Red17}. 

\begin{figure}[H]
\begin{center}
\includegraphics[width=13cm,clip,bb=0 0 507 252]{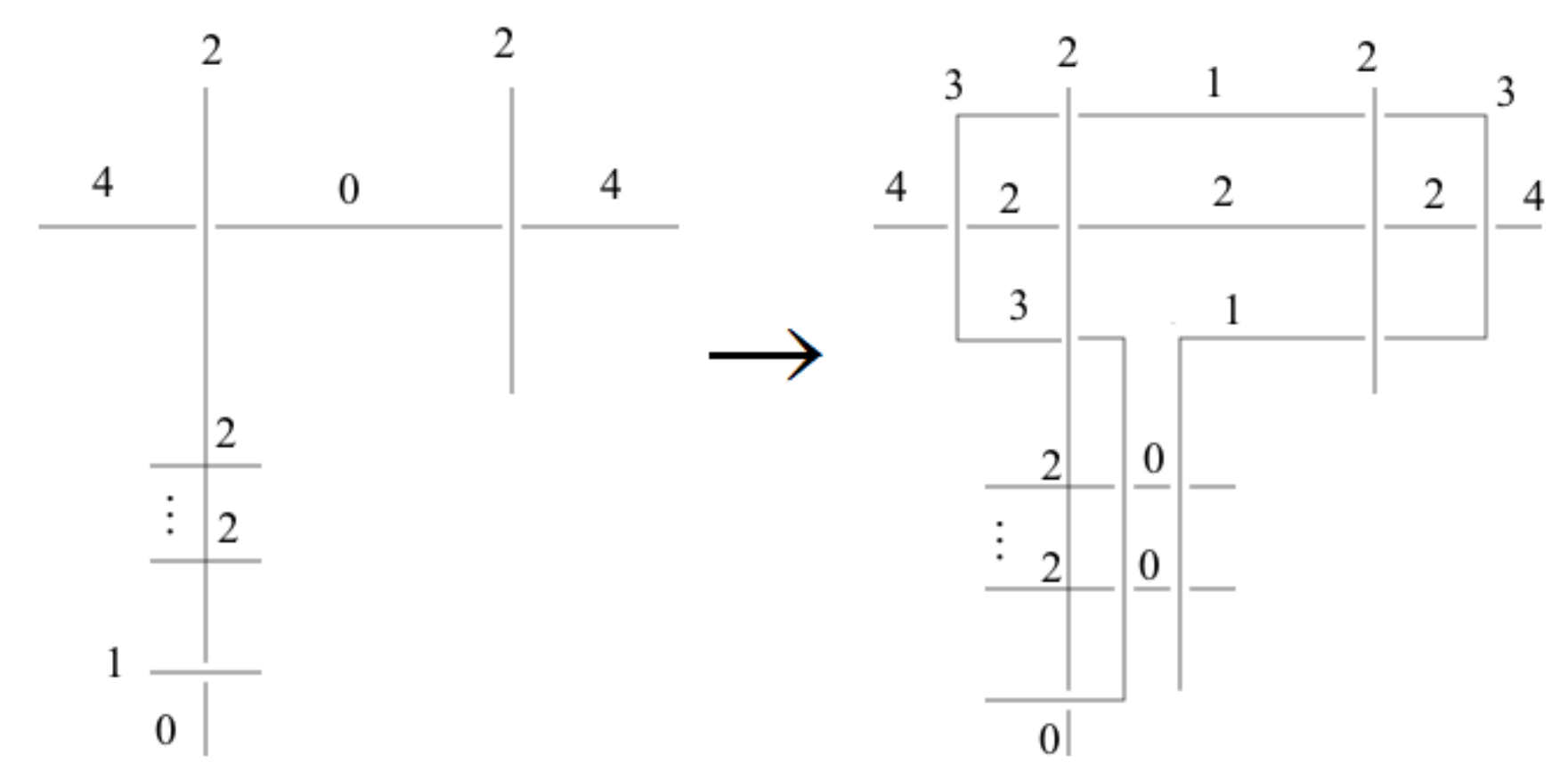}
\caption{}\label{Red17}
\end{center}
\end{figure}

\noindent
[2] In the case that, on the extension of the arc colored by $0$, 
a crossing colored by $\{0|1|2\}$ exists, 
then we transform the diagram as shown in Figure \ref{Red16}.

\begin{figure}[H]
\begin{center}
\includegraphics[width=12cm,clip,bb=0 0 765 294]{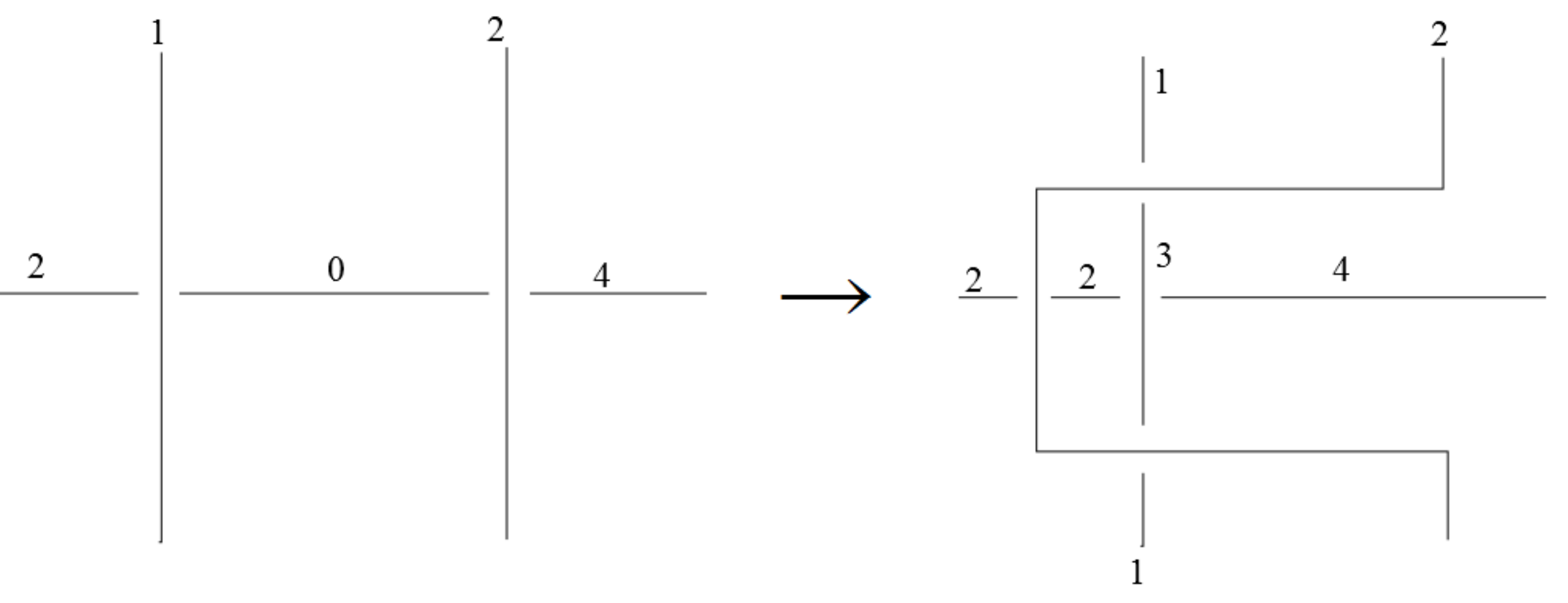}
\caption{}\label{Red16}
\end{center}
\end{figure}

After transformations [1] and [2], we obtain a palette graph as shown in Figure \ref{palette01247b}.

\begin{figure}[H]
\begin{center}
\includegraphics[width=4cm,clip,bb=0 0 302 385]{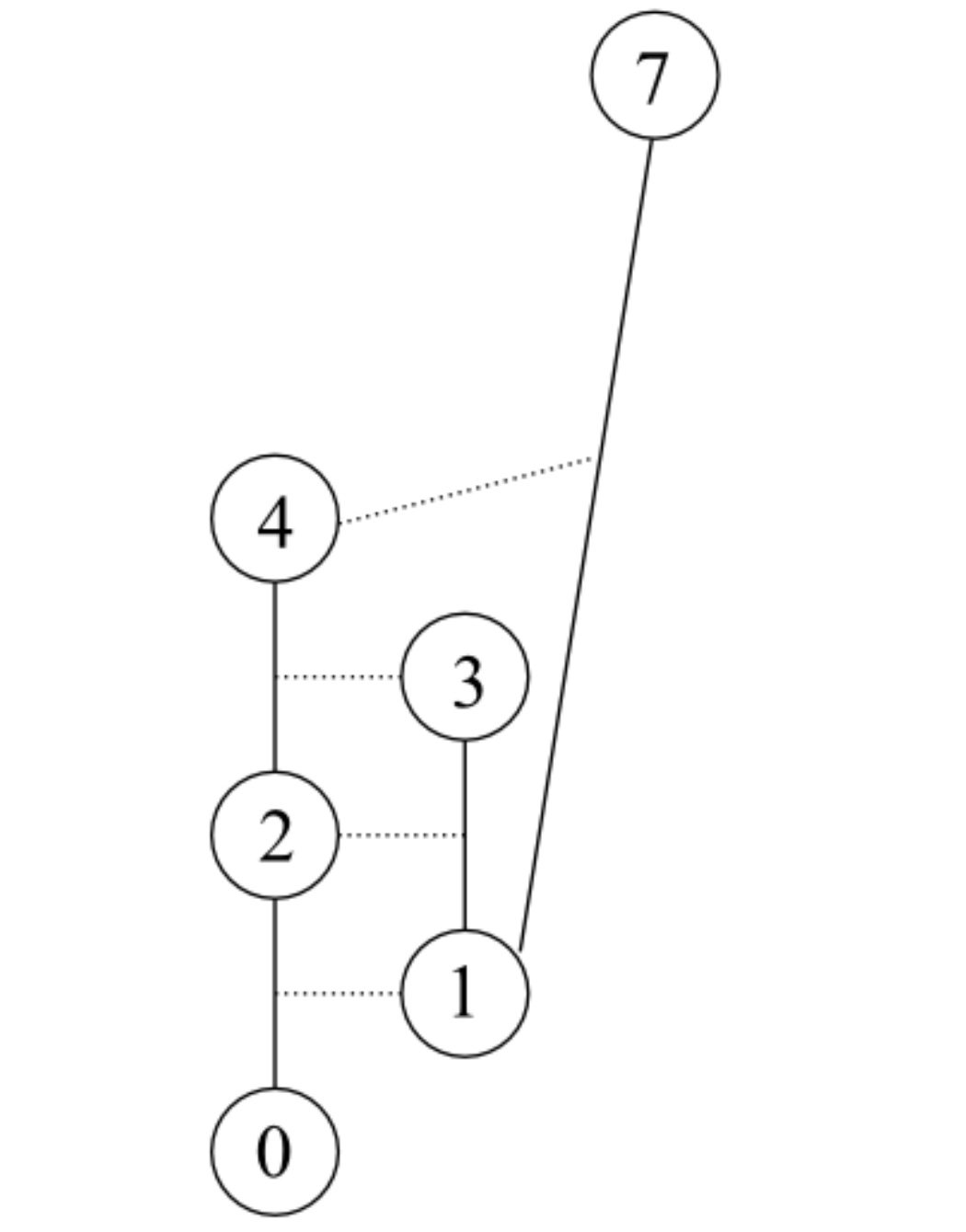}
\caption{}\label{palette01247b}
\end{center}
\end{figure}

\noindent
[3] We now delete a crossing colored by $\{7|4|1\}$. 
On the extension of the arc colored by $4$, 
a crossing colored by $\{4|3|2\}$ exists. 
Then we transform the diagram as shown in Figure \ref{Red18}.

\bigskip\bigskip

\begin{figure}[H]
\begin{center}
\includegraphics[width=13cm,clip,bb=0 0 606 384]{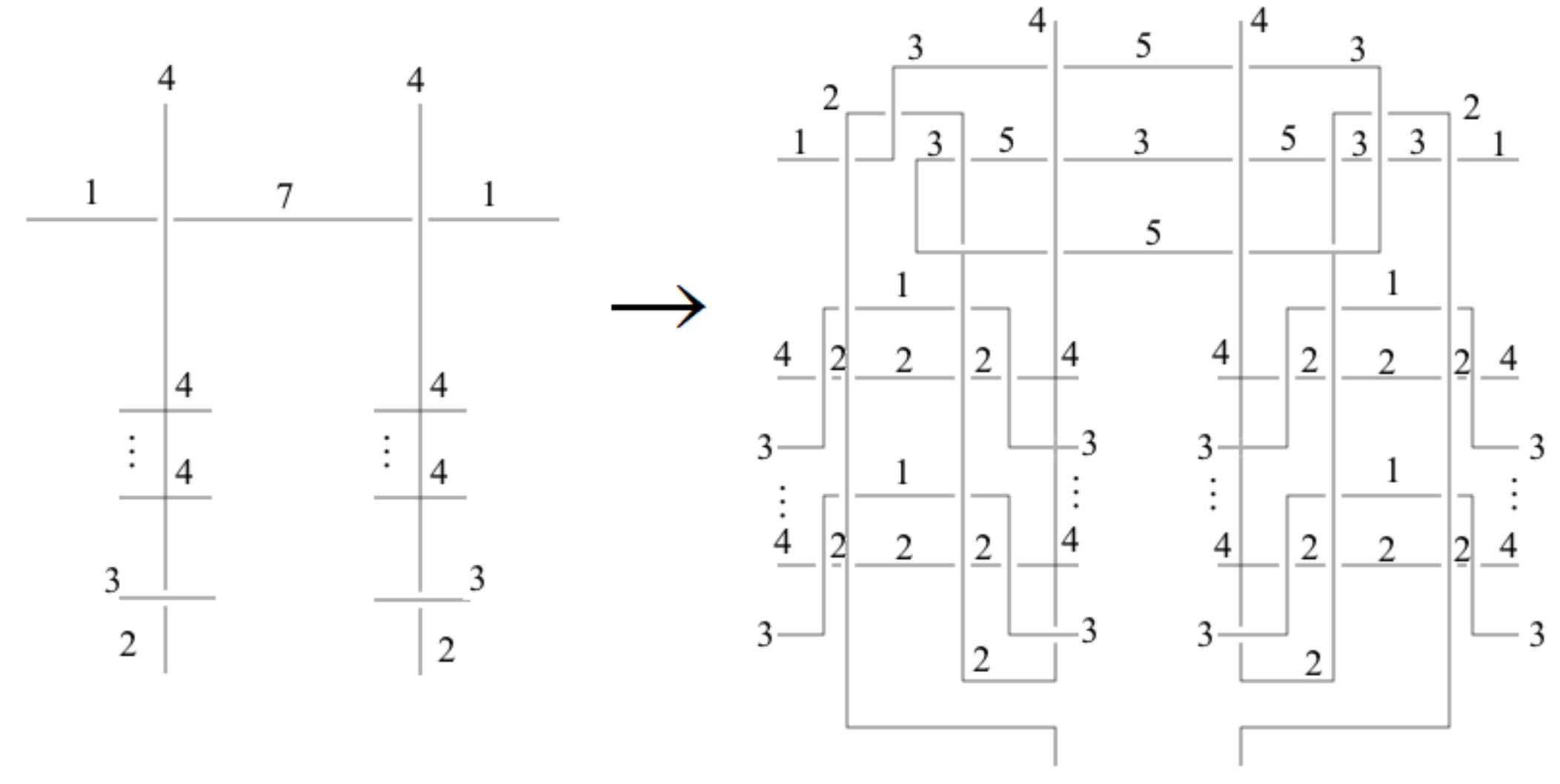}
\caption{}\label{Red18}
\end{center}
\end{figure}

Then we obtain another simple $\mathbb{Z}$-coloring $\gamma'$ such that {\it Im}$(\gamma')=\{0,1,2,3,4,5\}$.  
From Theorem 4.2, we see $mincol_\mathbb{Z}(L)=4$.\\

In the case that {\it Im}$(\gamma)=\{0,3,5,6,7\}$, 
we add $-7$ and multiply $-1$ to all colors for $\gamma$. 
Then this case is reduced to the case of {\it Im}$(\gamma)=\{0,1,2,4,7\}$.

Consequently we have completed our proof of Theorem \ref{thm42}. 
\end{proof}

\section{Links with simple $\mathbb{Z}$-colorings}

In this section, we present diagrams of the links with at most 10 crossings with simple $\mathbb{Z}$-colorings. 
We first use the diagrams given in \cite{linkinfo}. For each link, if the diagram admits
only a non-simple $\mathbb{Z}$-coloring, then we modify it to admit a simple one.

\begin{figure}[H]
\begin{center}
\includegraphics[width=10cm,clip,bb = 0 0 457 341]{ap1.pdf}
\caption{$L8n6$}\label{L8n6}
\includegraphics[width=10cm,clip,bb=0 0 576 432]{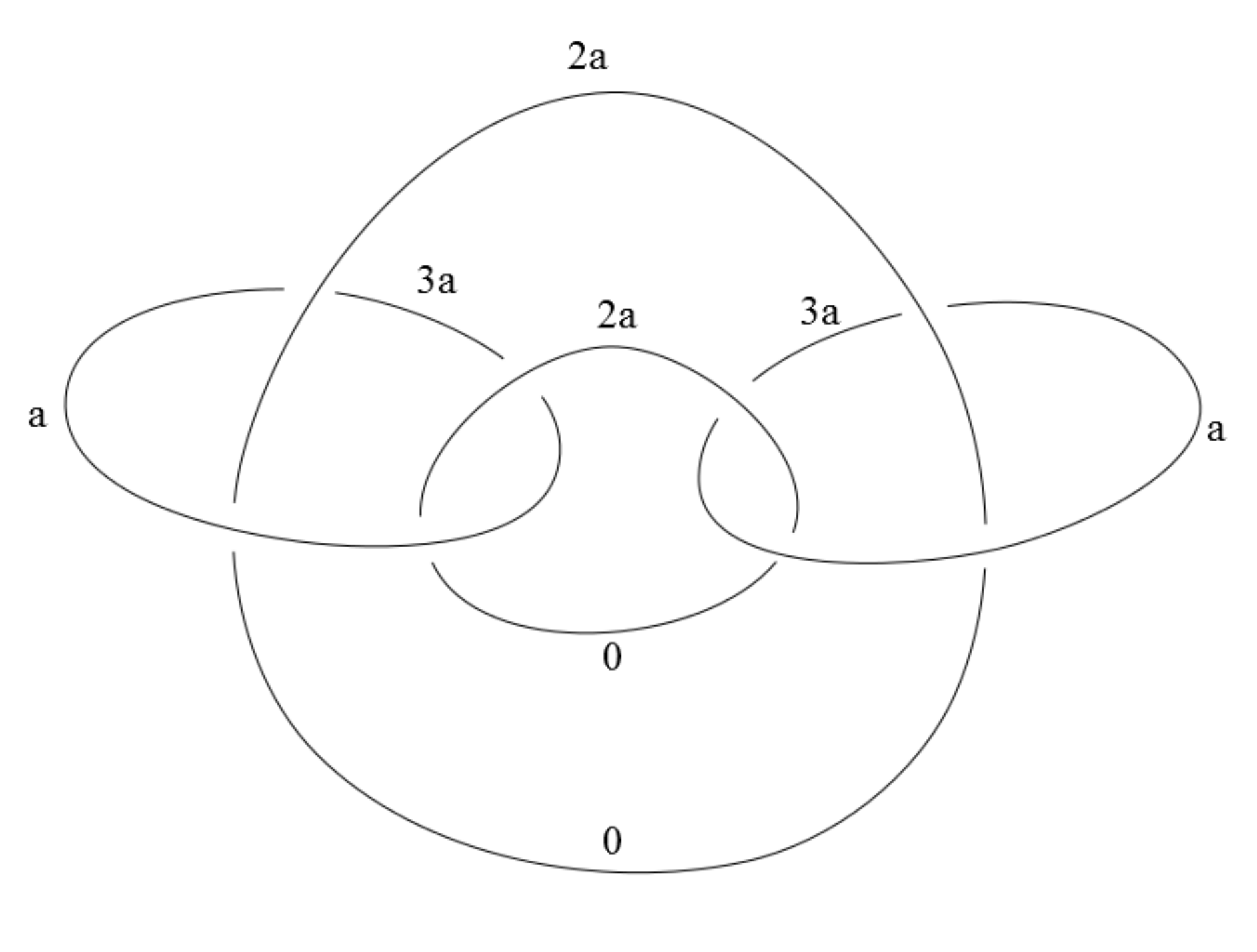}
\caption{$L8n8$}\label{L8n8}
\end{center}
\end{figure}

\begin{figure}[H]
\begin{center}
\includegraphics[width=9cm,clip,bb=0 0 502 383]{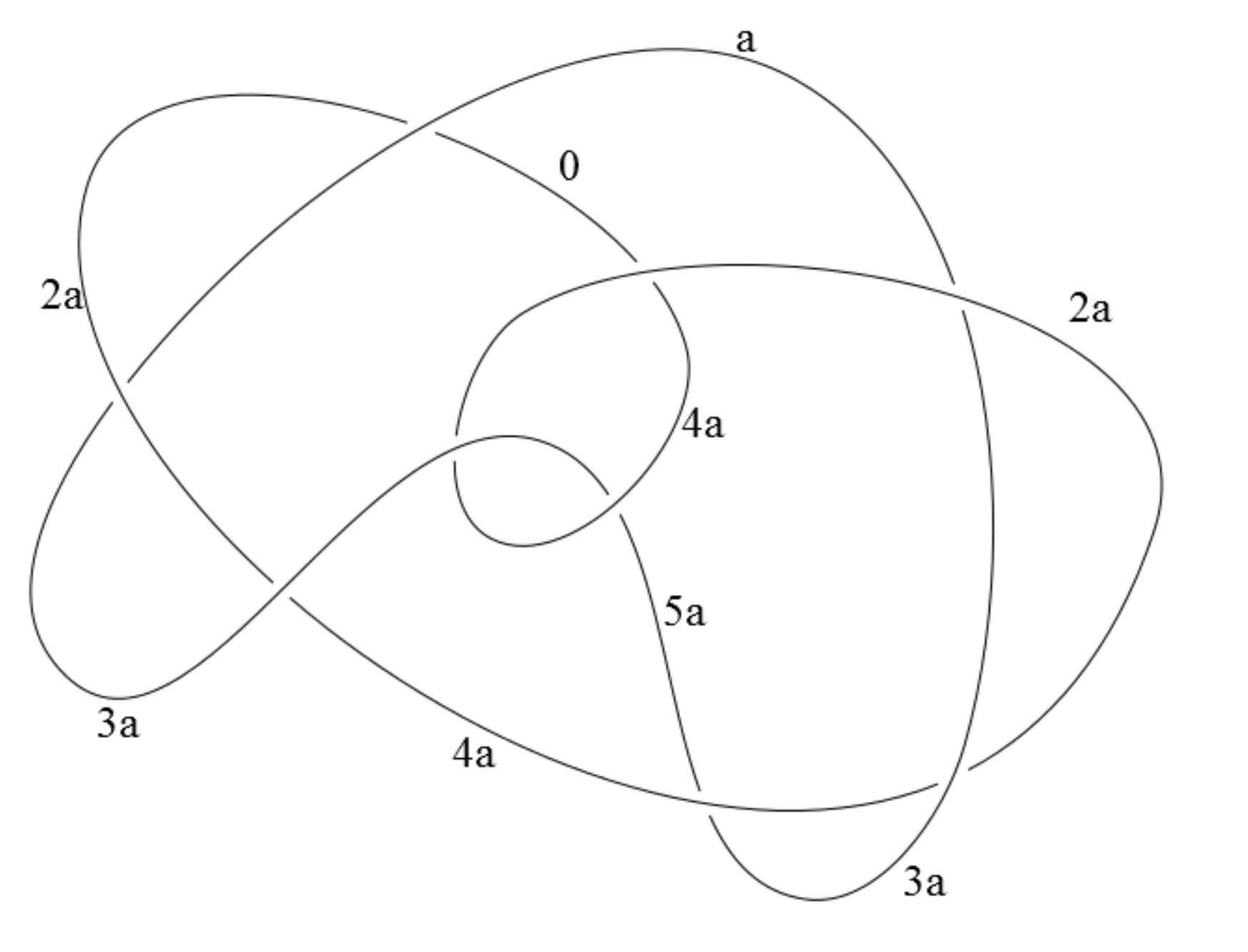}
\caption{$L9n18$}\label{L9n18}
\includegraphics[width=7.8cm,clip,bb=0 0 404 410]{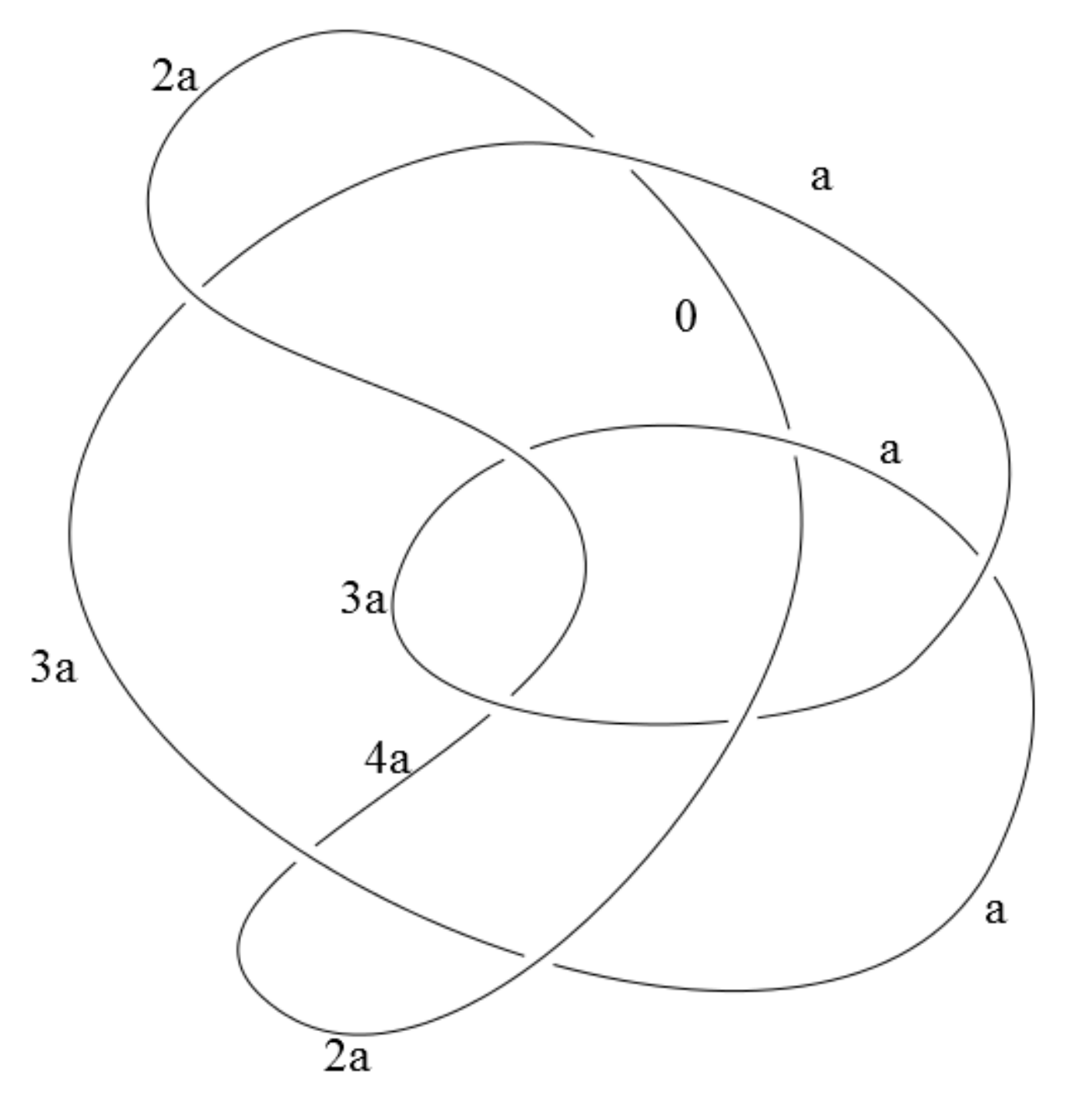}
\caption{$L9n19$}\label{L9n19}
\end{center}
\end{figure}

\begin{figure}[H]
\begin{center}
\includegraphics[width=7cm,clip,bb=0 0 415 376]{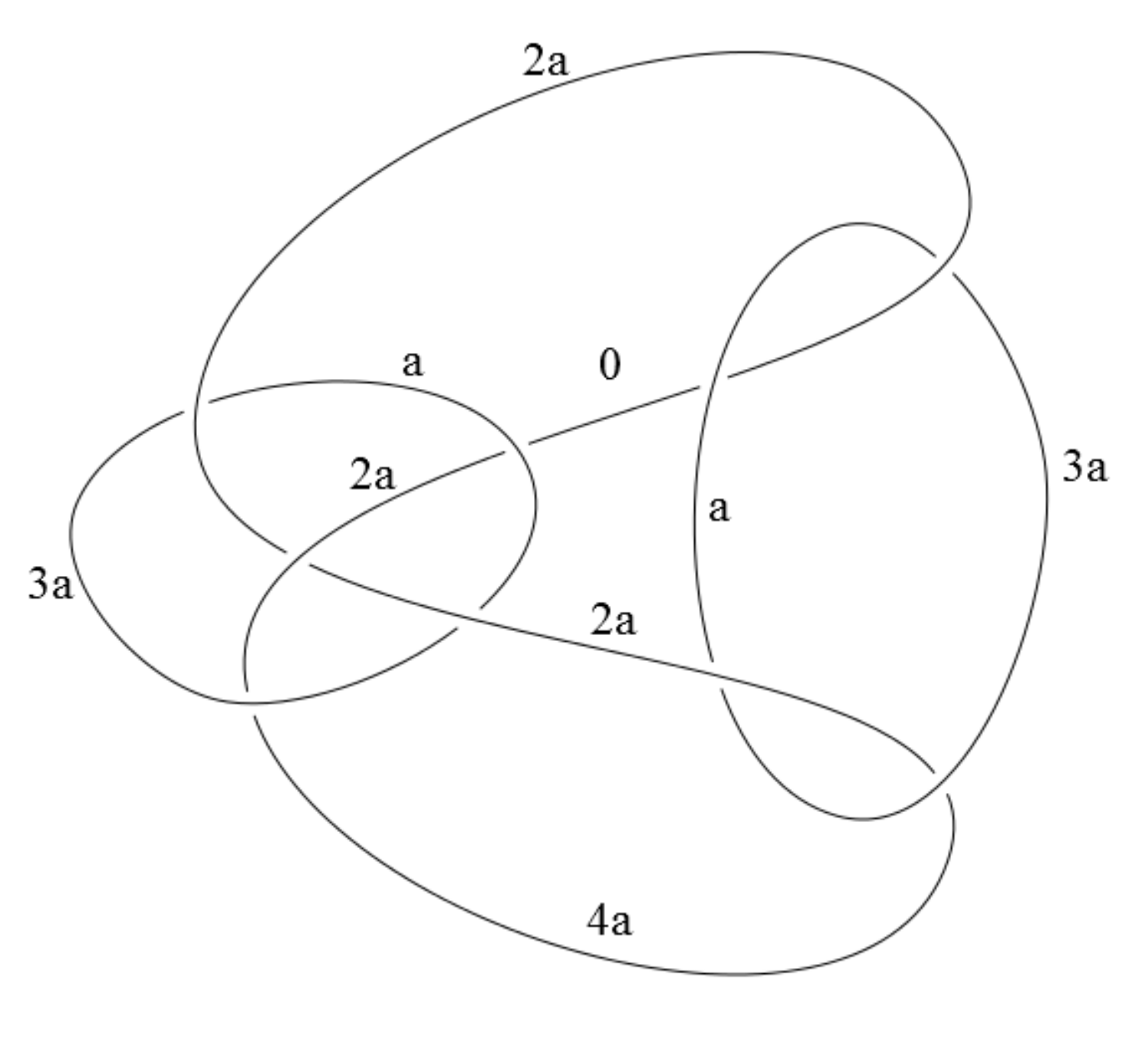}
\caption{$L9n27$}\label{L9n27}
%non
\includegraphics[width=7cm,clip,bb=0 0 453 414]{ap6.pdf}
\caption{$L10n32$ with non-simple $\mathbb{Z}$-coloring}\label{L10n32}
\includegraphics[width=8cm,clip,bb=0 0 600 450]{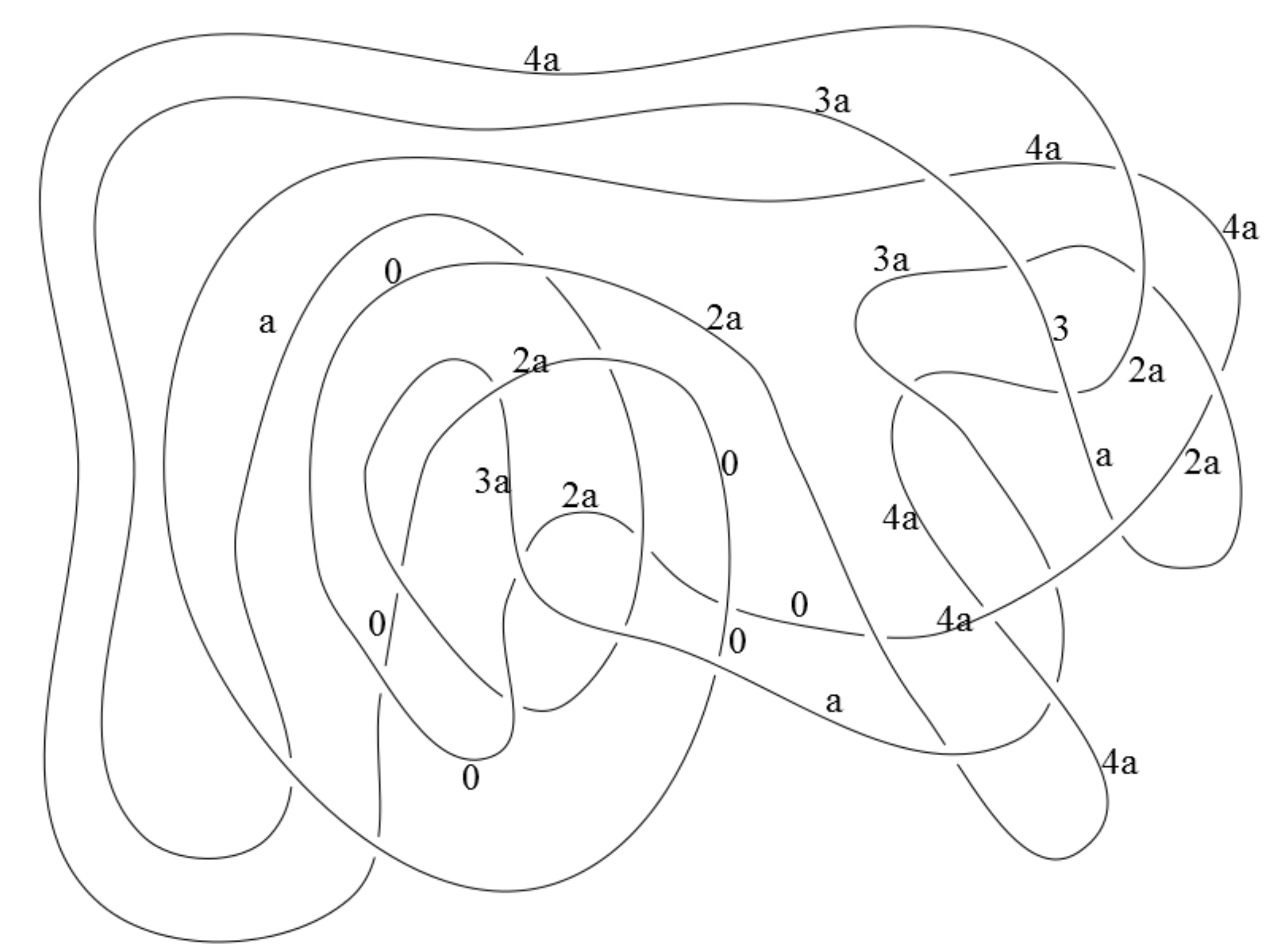}
\caption{$L10n32$ with simple $\mathbb{Z}$-coloring}\label{L10n32-2}
\end{center}
\end{figure}

\begin{figure}[H]
\begin{center}
%non
\includegraphics[width=7cm,clip,bb=0 0 464 375]{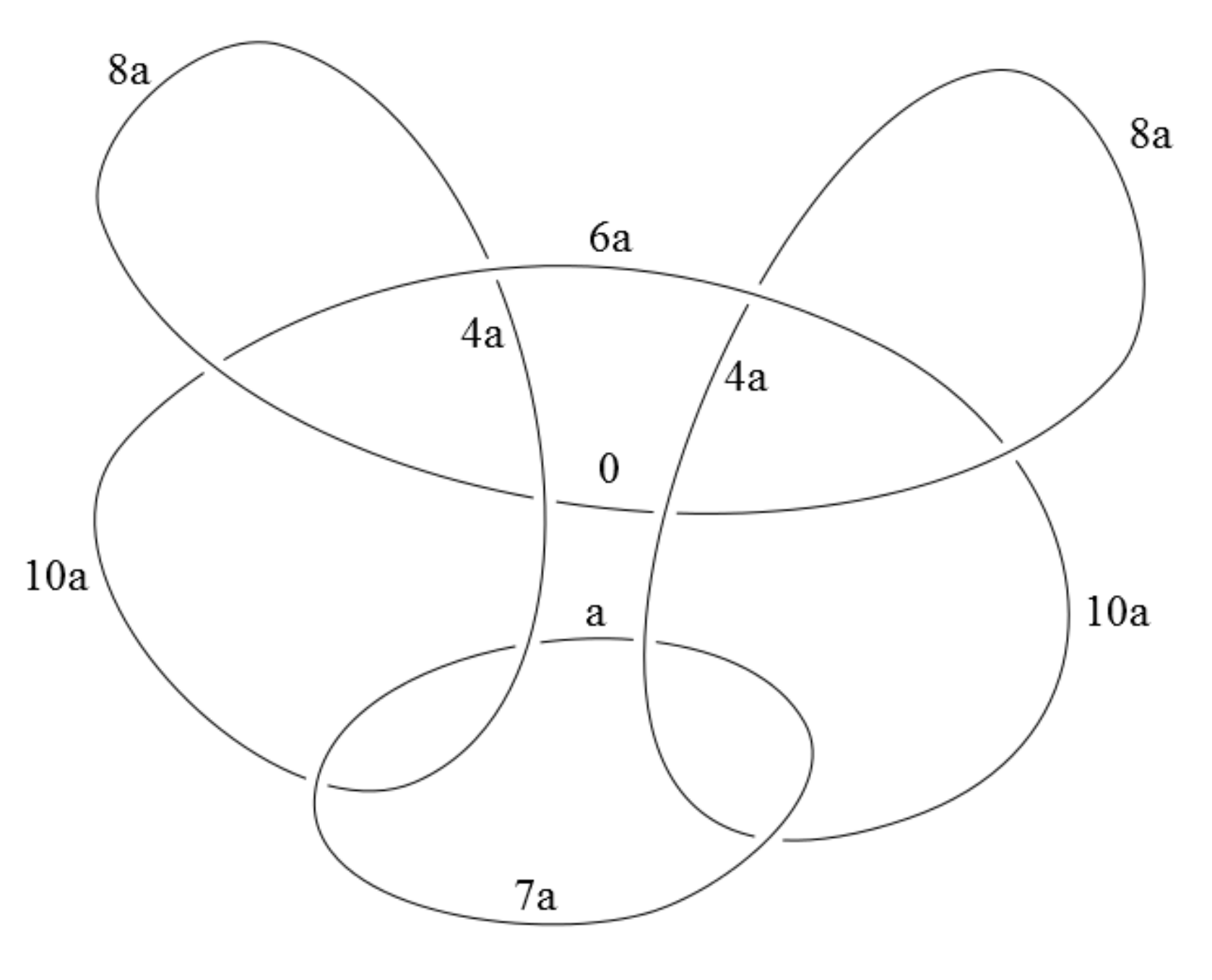}
\caption{$L10n36$}\label{L10n36}
\includegraphics[width=9cm,clip,bb=0 0 600 450]{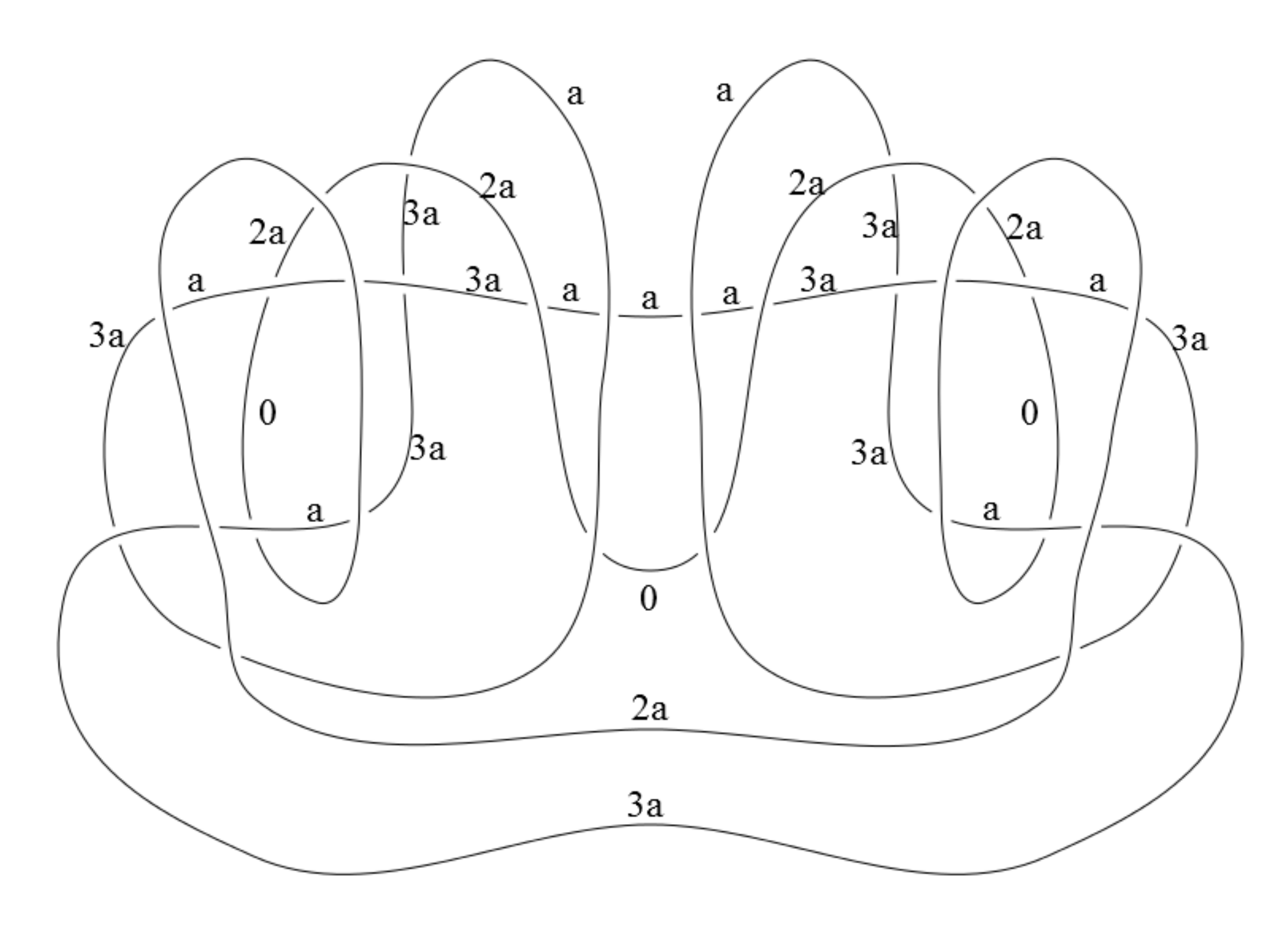}
\caption{$L10n36$ with simple $\mathbb{Z}$-coloring}\label{L10n36-2}
\includegraphics[width=8cm,clip,bb=0 0 489 397]{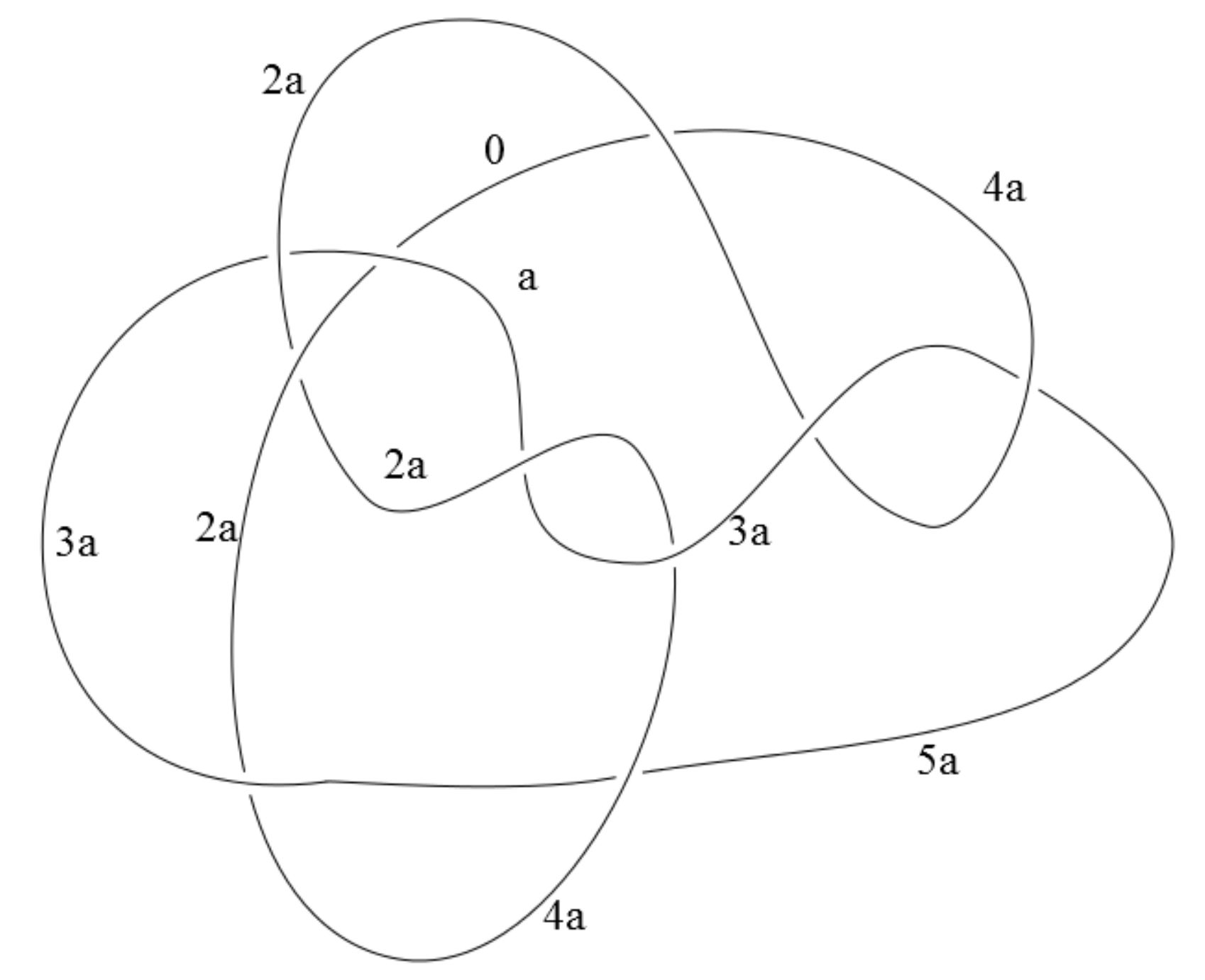}
\caption{$L10n56$}\label{L10n56}
\end{center}
\end{figure}

\begin{figure}[H]
\begin{center}
\includegraphics[width=7cm,clip,bb=0 0 427 373]{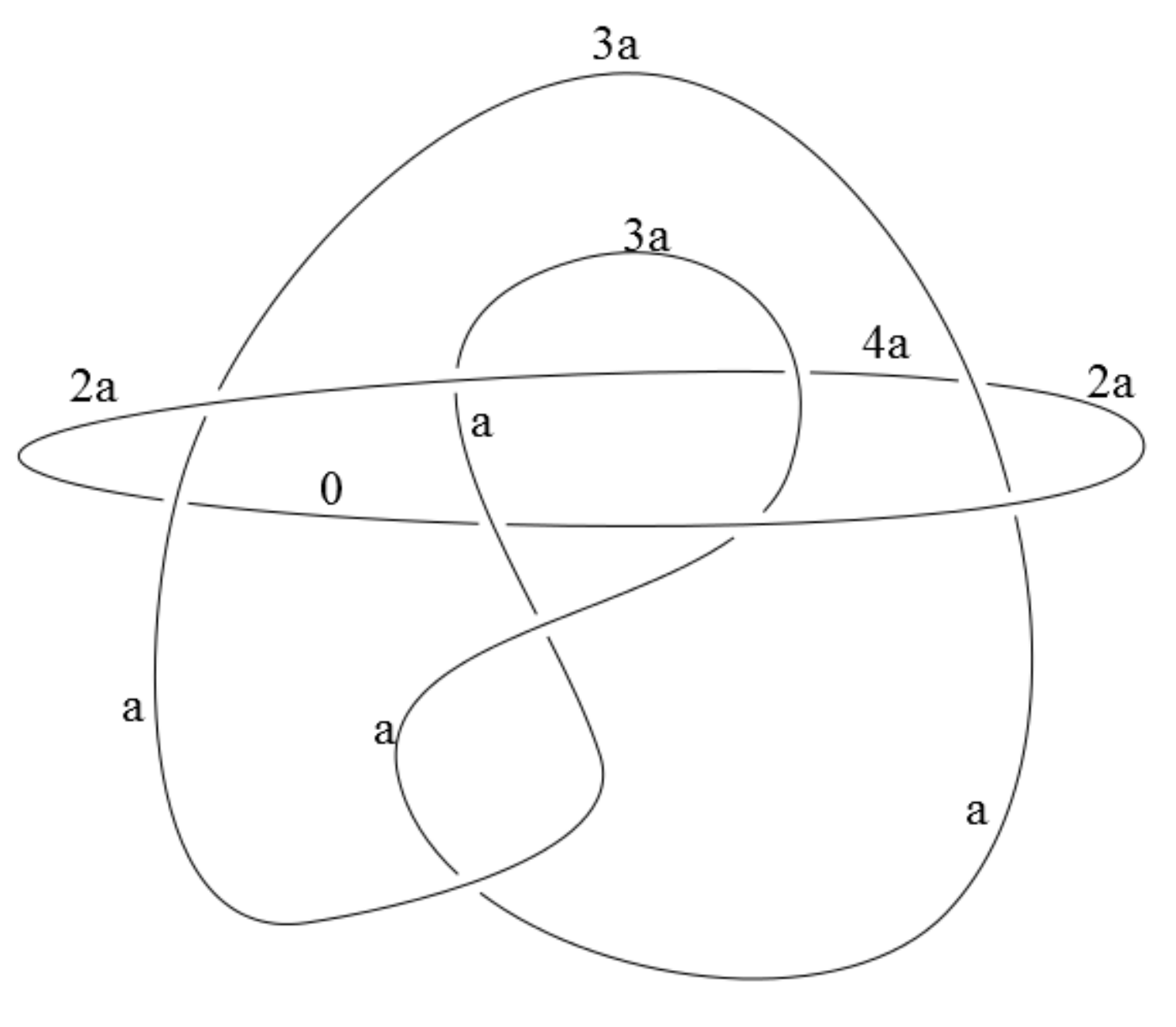}
\caption{$L10n57$}\label{L10n57}
%non
\includegraphics[width=9cm,clip,bb=0 0 560 362]{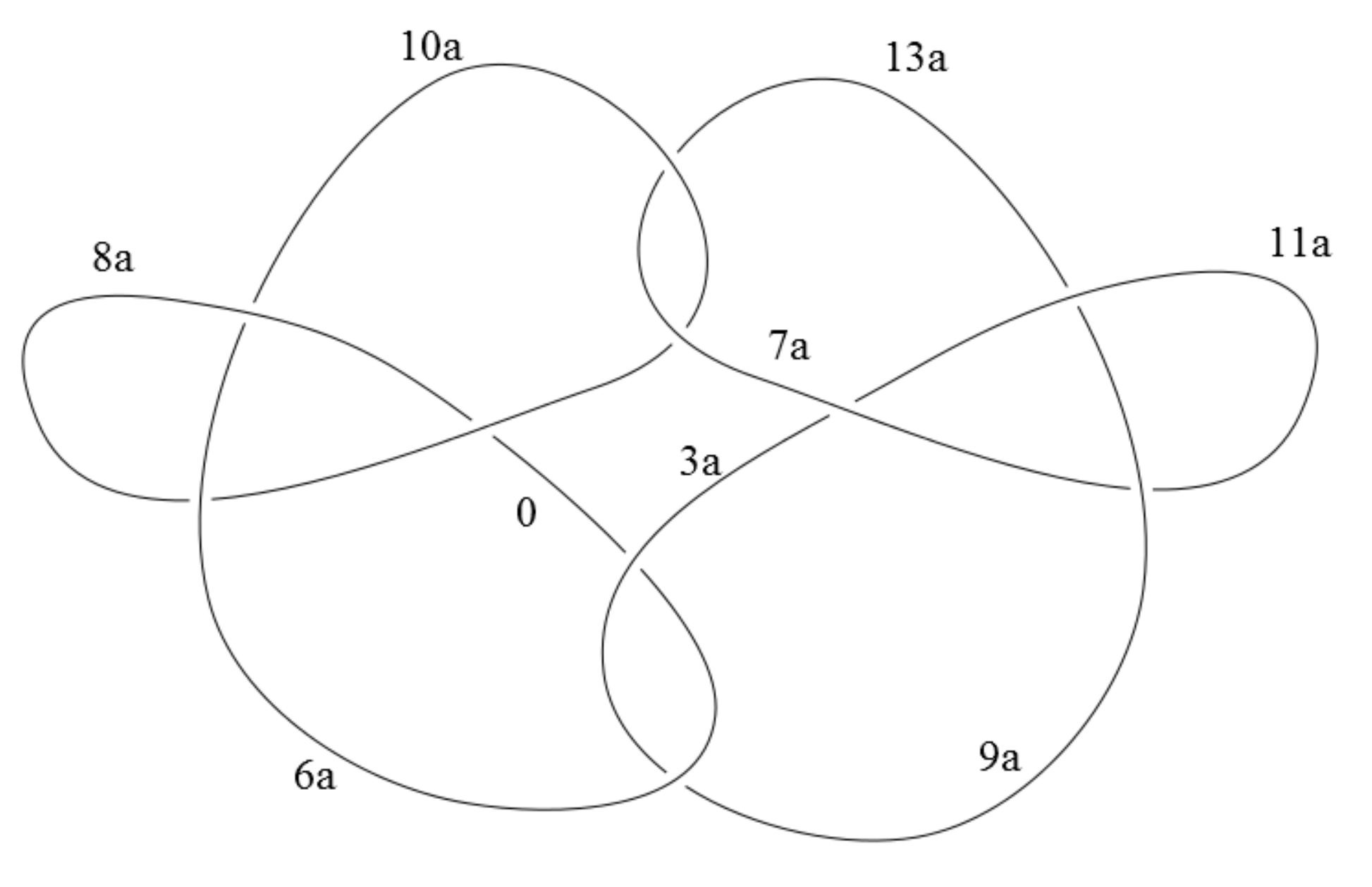}
\caption{$L10n59$ with non-simple $\mathbb{Z}$-coloring}\label{L10n59}
\includegraphics[width=9cm,clip,bb=0 0 600 450]{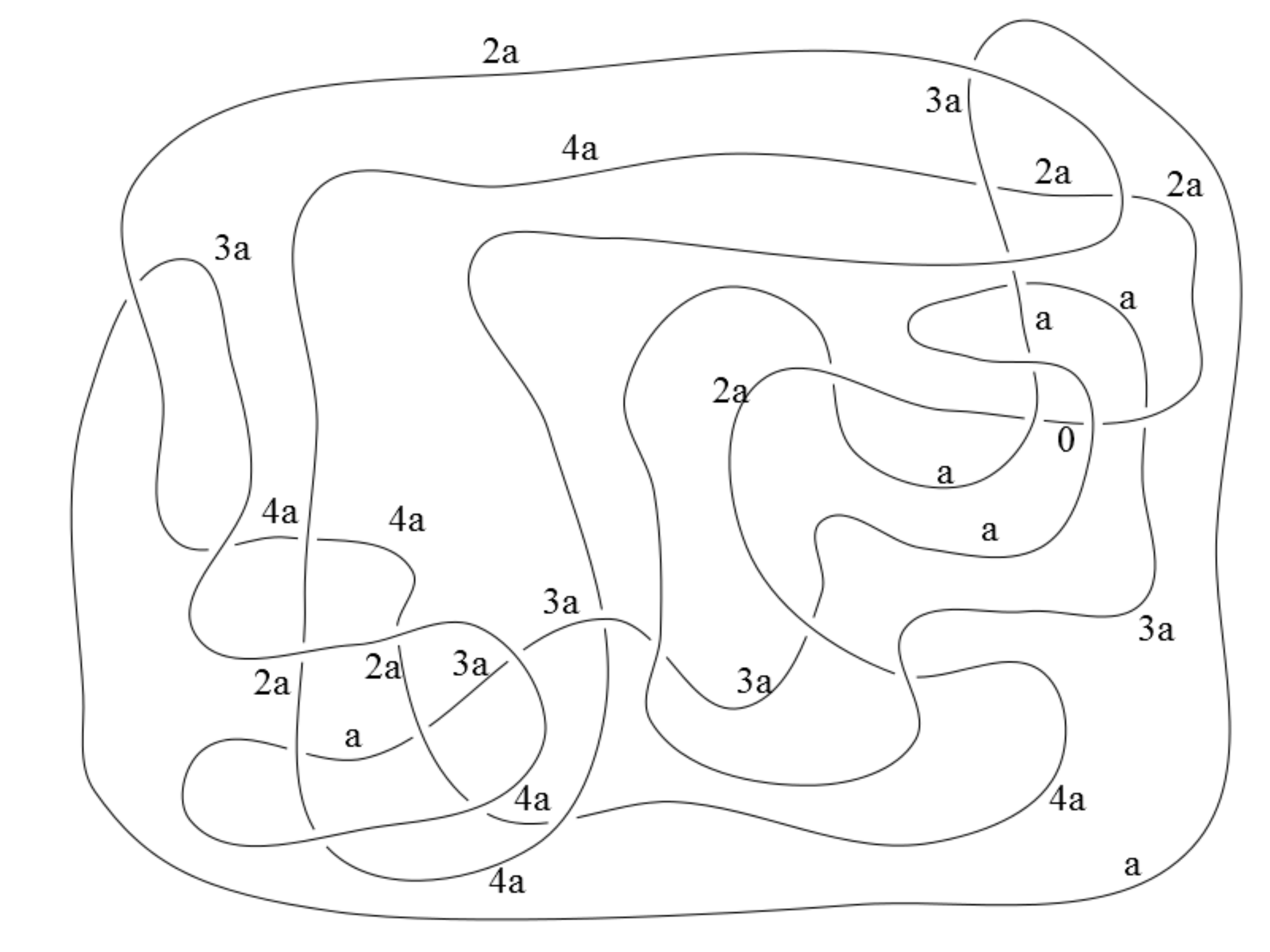}
\caption{$L10n59$ with simple $\mathbb{Z}$-coloring}\label{L10n59-2}
\end{center}
\end{figure}

\begin{figure}[H]
\begin{center}
\includegraphics[width=8cm,clip,bb=0 0 600 450]{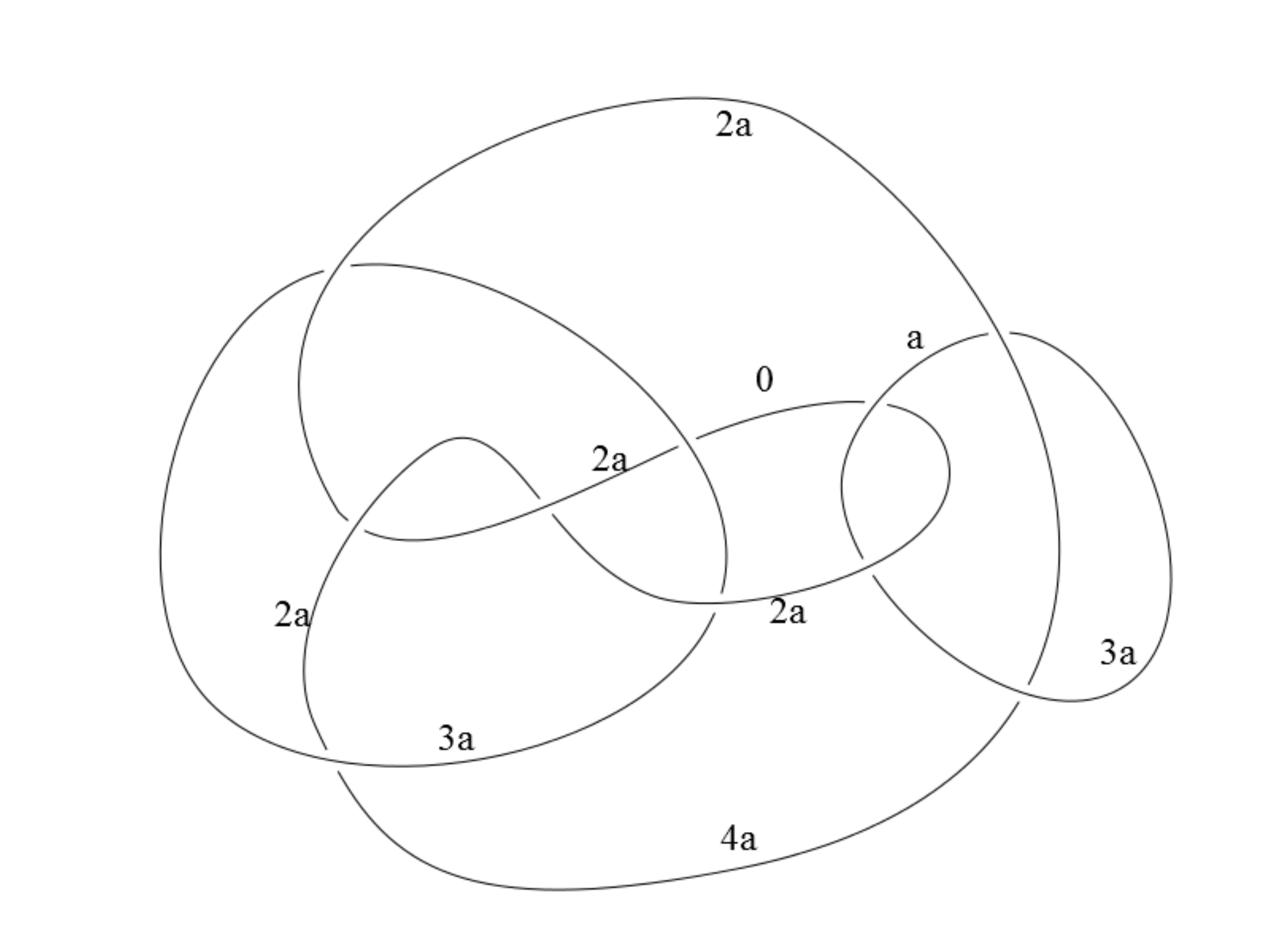}
\caption{$L10n91$}\label{L10n91}
\includegraphics[width=8cm,clip,bb=0 0 490 340]{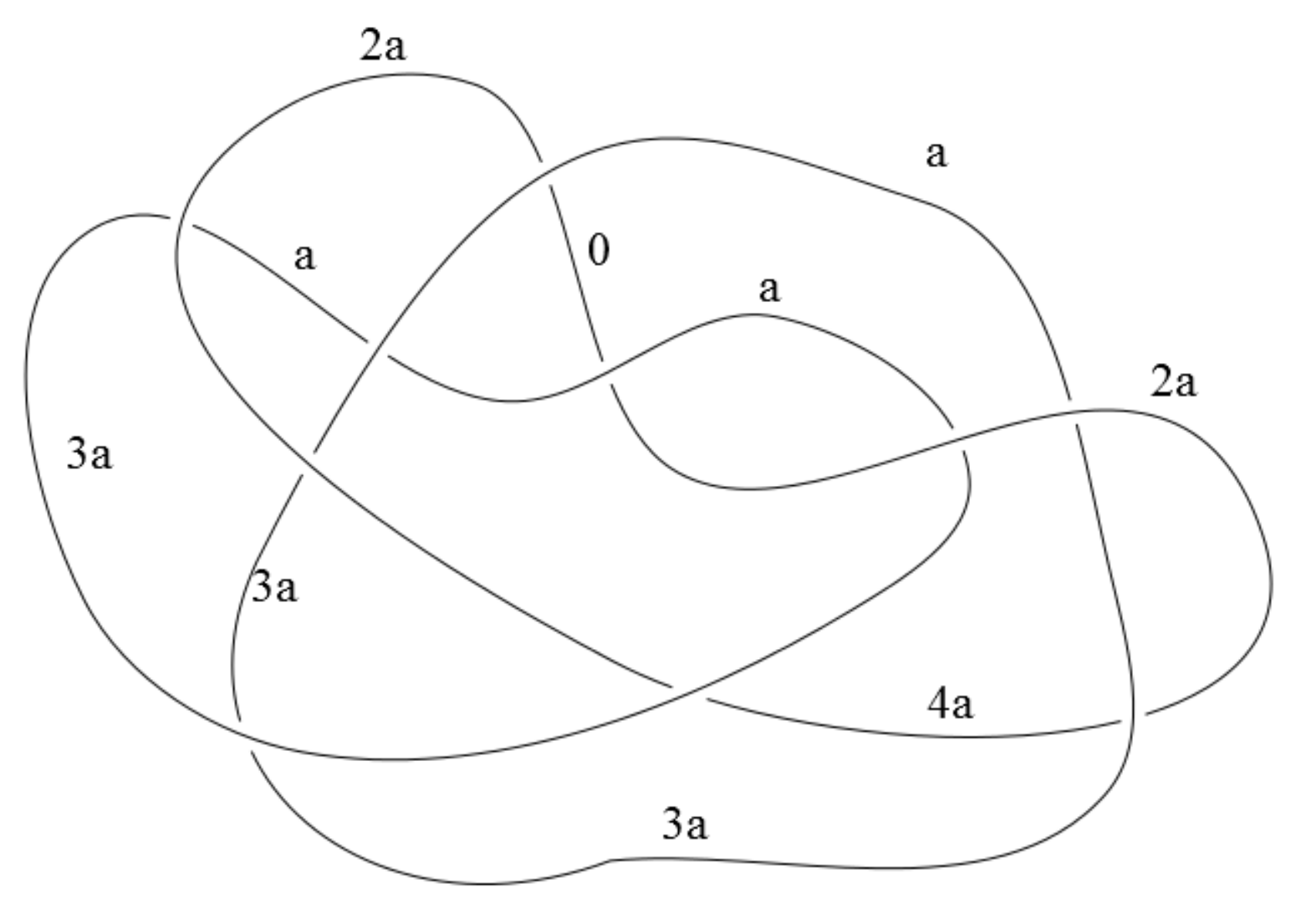}
\caption{$L10n93$}\label{L10n93}
\includegraphics[width=8cm,clip,bb=0 0 412 390]{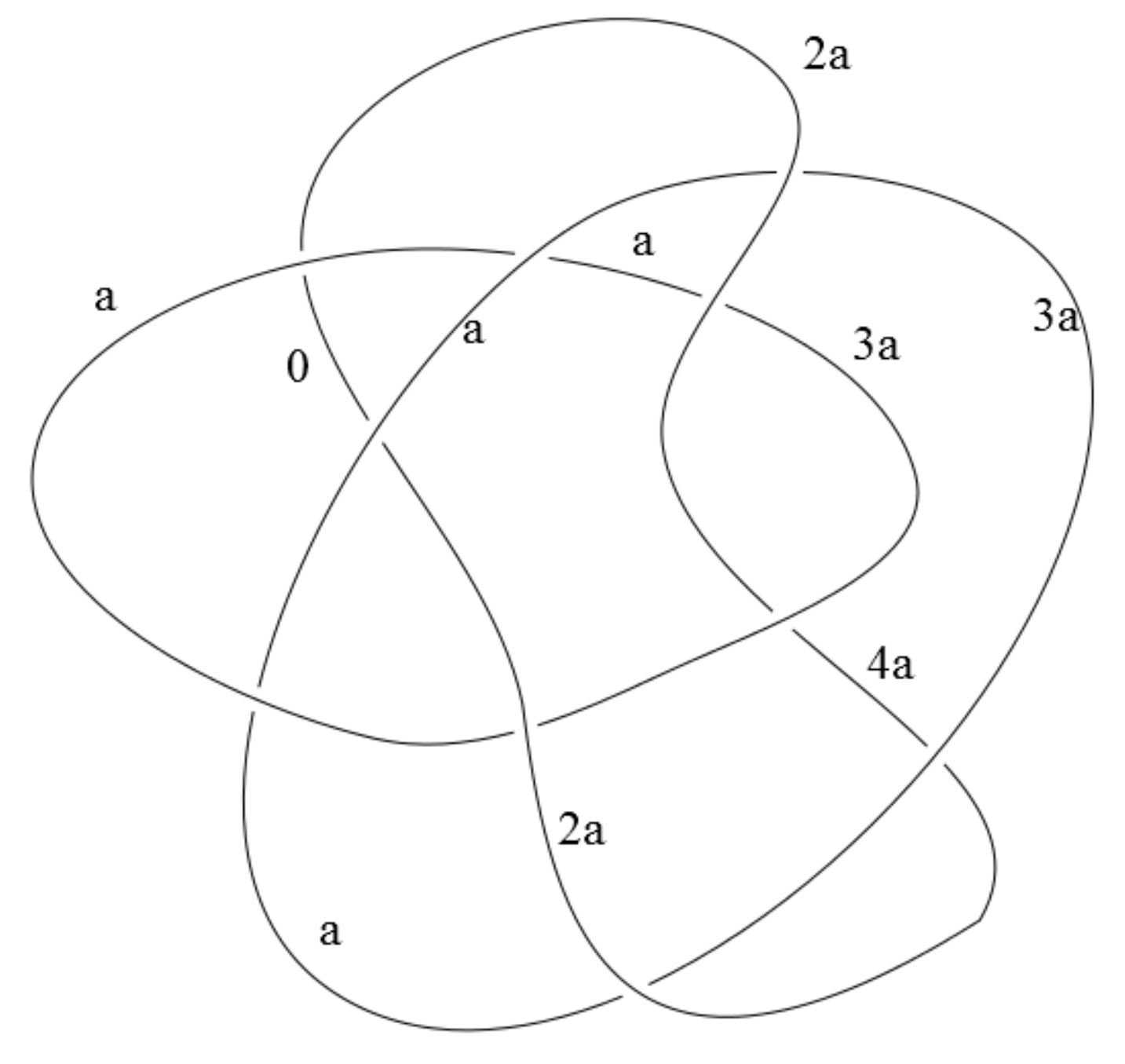}
\caption{$L10n94$}\label{L10n94}
\end{center}
\end{figure}

\begin{figure}[H]
\begin{center}
\includegraphics[width=8cm,clip,bb=0 0 442 346]{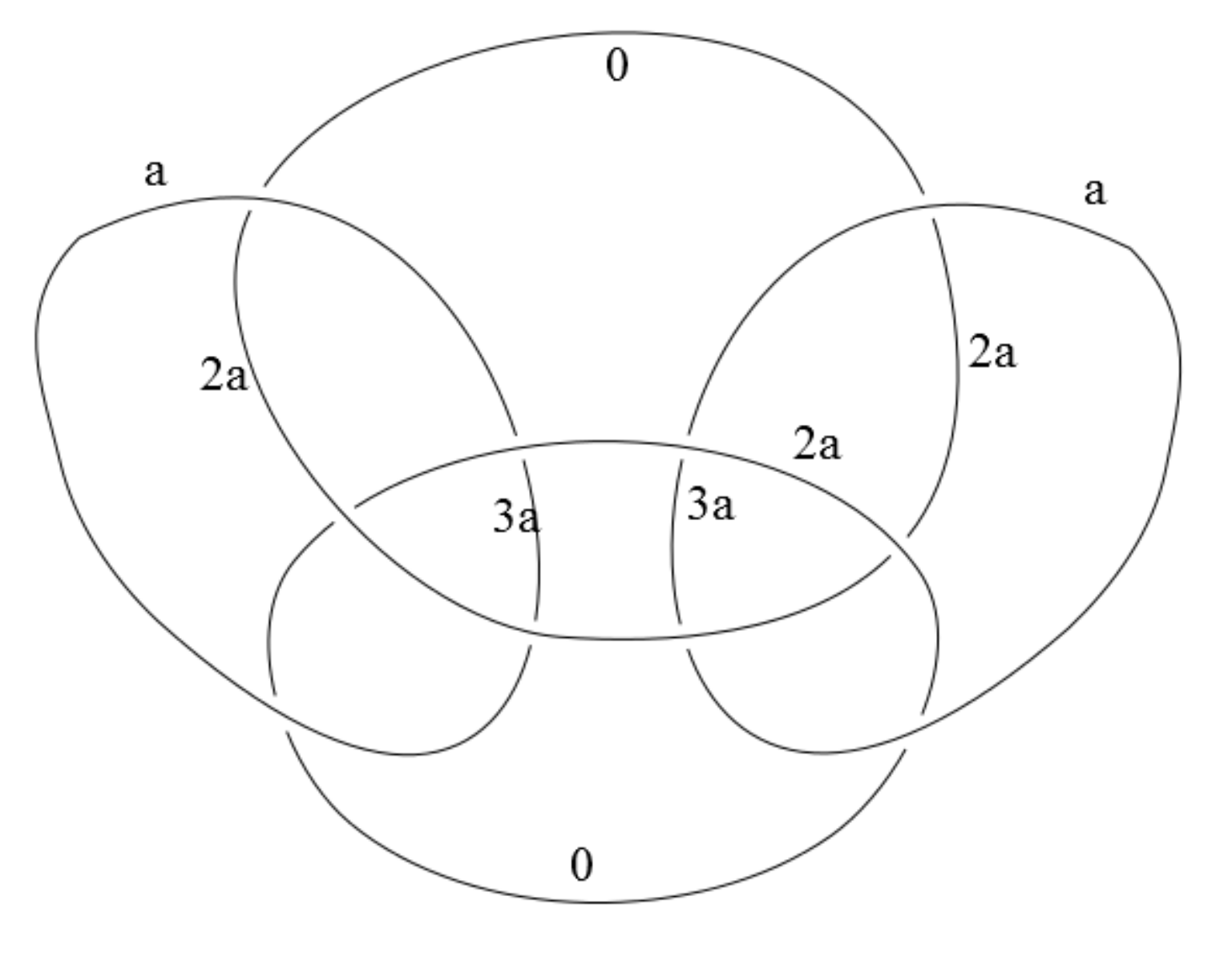}
\caption{$L10n104$}\label{L10n104}
\includegraphics[width=8cm,clip,bb=0 0 432 346]{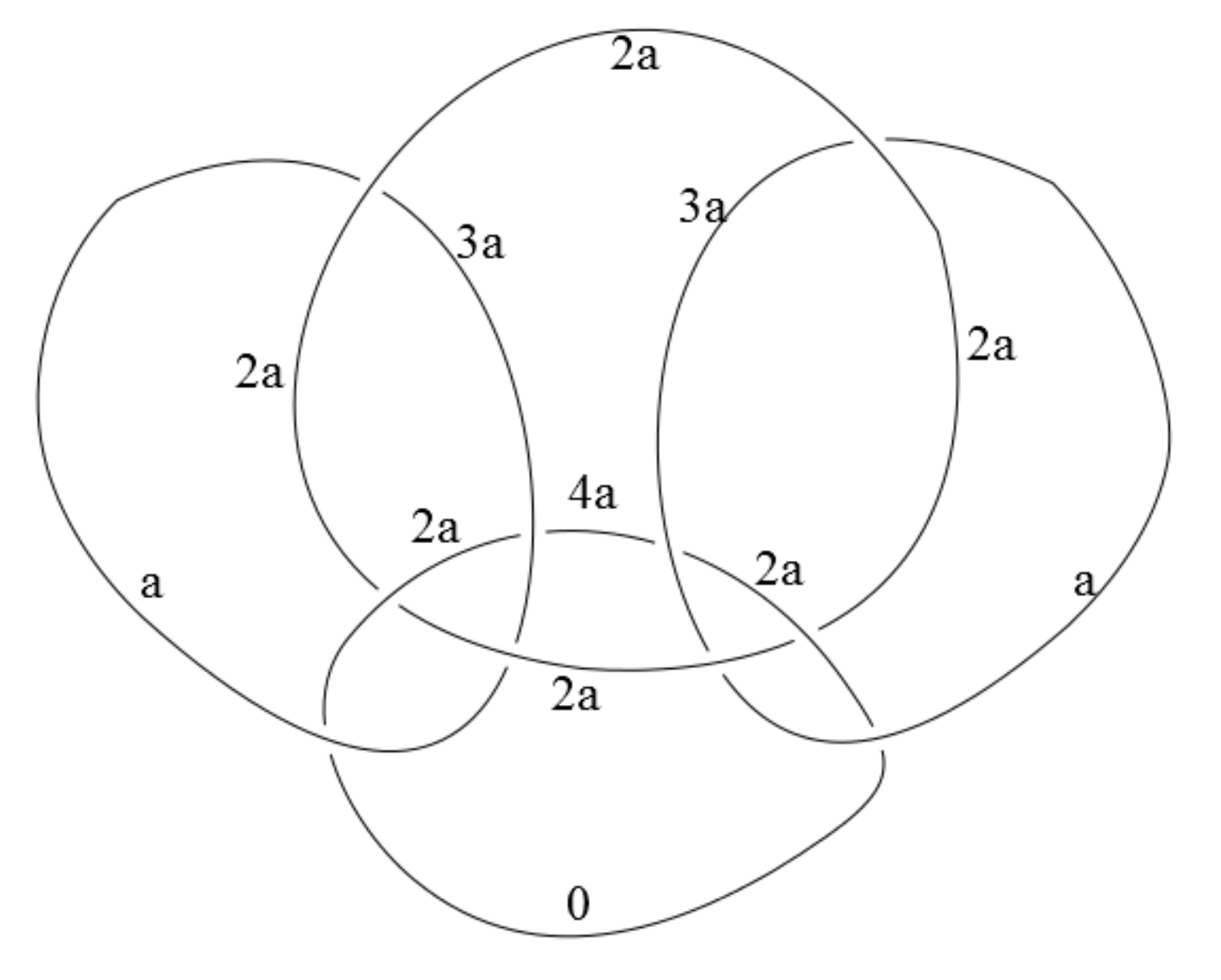}
\caption{$L10n107$}\label{L10n107}
\includegraphics[width=8cm,clip,bb=0 0 459 360]{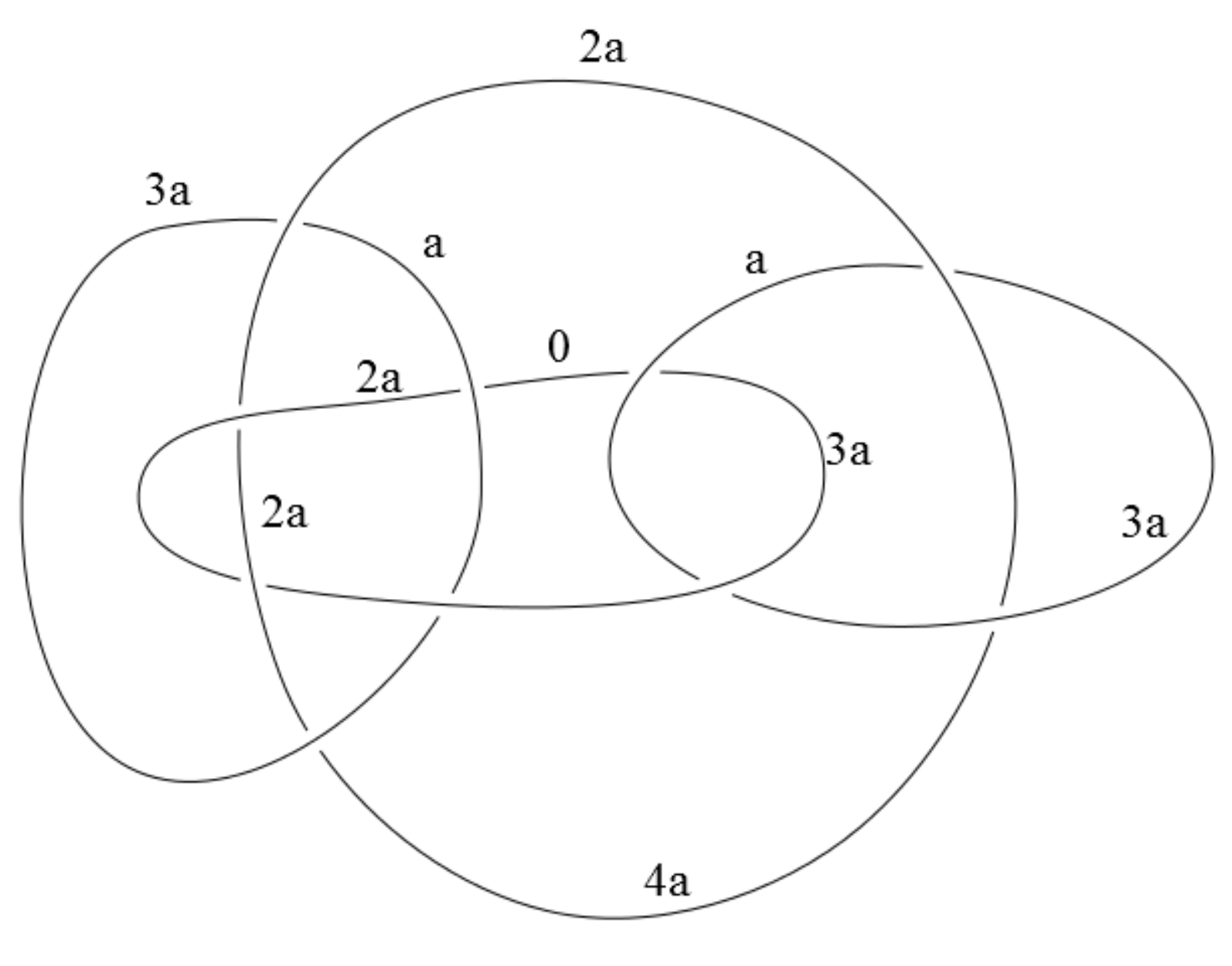}
\caption{$L10n111$}\label{L10n111}
\end{center}
\end{figure}

%%%%%%%%%%%%%%%%%%%%%%%%%%%%%%%%%
\section{Question}\label{sec:que}

In view of our studies in the previous sections, 
it is natural to ask;

\begin{question}\label{quest1}
Does mincol$_\mathbb{Z}(L)=4$ always hold for any non-splittable $\mathbb{Z}$-colorable link $L$?
\end{question}

This is equivalent to the next by virtue of Theorem \ref{simplethm}.

\begin{question}\label{quest2}
Does every non-splittable $\mathbb{Z}$-colorable link admit a simple $\mathbb{Z}$-coloring?
\end{question}

%%%%%%%%%%%%%%%%%%%%%%%%%%%%%%%%%
\section*{Acknowledgement}
The authors would like to thank Shin Satoh for his useful comments. 
The first author is partially supported by JSPS KAKENHI 
Grant Number 26400100.

%%%%%%%%%%%%%%%%%%%%%%%%%%%%%%%%%

\end{document}